\documentclass[11pt]{amsart}

\usepackage{enumitem}
\usepackage[ddmmyyyy]{datetime}
\usepackage{geometry}
\usepackage{lineno,hyperref}
\usepackage{float}
\usepackage{amsfonts}
\usepackage{mathrsfs}
\usepackage{amssymb}
\usepackage{amsmath}
\usepackage{amsthm}
\usepackage{graphicx}
\usepackage[font=small,labelfont=bf]{caption}
\usepackage{subcaption}
\usepackage{cite}
\usepackage[usenames,dvipsnames]{xcolor}
\numberwithin{equation}{section}

\graphicspath{{Figs}}%

\newcommand{\bsa}{{\boldsymbol{a}}}
\newcommand{\bsx}{{\boldsymbol{x}}}

\newcommand{\bsy}{{\boldsymbol{y}}}
\newcommand{\bsz}{{\boldsymbol{z}}}

\newcommand{\bsX}{{\boldsymbol{X}}}
\newcommand{\bbR}{\mathbb{R}}

\newcommand{\Cqmc}{C^{\textsc{qmc}}}
\newcommand{\Cube}{\mathcal{Q}}
\newcommand{\rd}{\mathrm{d}}
\newcommand{\IP}[2]{#1 \cdot #2}
\newcommand{\LBo}[1]{\Delta_{#1}^*} 
\newcommand{\filter}[1]{h(#1)} 
\newcommand{\ffilter}[1]{h\Big(#1\Big)} 
\DeclareMathOperator{\supp}{supp}
\newcommand{\floor}[1]{\left\lfloor{#1}\right\rfloor}

\DeclareMathOperator{\wce}{wce}
\DeclareMathOperator{\dist}{dist}

\newtheorem{theorem}{Theorem}
\numberwithin{theorem}{section}

\theoremstyle{definition}
\newtheorem{definition}[theorem]{Definition}

\newtheorem*{remark*}{Remark}

\date{\today}

\graphicspath{{Figs/}}

\begin{document}

\title{Needlets Liberated}

\author[J. S. Brauchart]{Johann S. Brauchart\textsuperscript{\textasteriskcentered}\textsuperscript{\dag}}
\author[P. J. Grabner]{Peter J. Grabner\textsuperscript{\ddag}}
\author[I. H. Sloan]{Ian H. Sloan}
\author[R. S. Womersley]{Robert S. Womersley}

\address[J. S. B., P. J. G.]{Institute of Analysis and Number Theory, Graz University of Technology, 8010 Graz, Austria}
\address[I. H. S., R. S. W.]{School of Mathematics and Statistics, UNSW Australia, Sydney, NSW, 2052, Australia}

\thanks{\noindent \textsuperscript{\textasteriskcentered}Corresponding author. \textsuperscript{\dag}The research of this author was supported by the Austrian Science Fund FWF Lise Meitner Scholarship project M2030 ``Self organization by local interaction''.}

\thanks{\textsuperscript{\ddag}This author was supported by the Austrian Science Fund FWF project F5503 (part of the Special Research Program (SFB) ``Quasi-Monte Carlo Methods: Theory and Applications'').}

\thanks{This research was supported under the Australian Research Council's \emph{Discovery Project} DP180100506}.

\thanks{The first three authors also acknowledge the support of the Erwin Schr{\"o}dinger Institute in Vienna, where part of the work was carried out during the programme ``Tractability of High Dimensional Problems and Discrepancy'' and all authors also acknowledge the support of the National Science Foundation under Grant No. DMS-1439786 while they were in residence at the Institute for Computational and Experimental Research in Mathematics (ICERM) in Providence, RI, during the Spring 2018 Semester Program on "Point Configurations in Geometry, Physics and Computer Science".}

\email{j.brauchart@tugraz.at}
\email{peter.grabner@tugraz.at}
\email{i.sloan@unsw.edu.au}
\email{r.womersley@unsw.edu.au}

\begin{abstract}
  Spherical needlets were introduced by Narcowich, Petrushev, and Ward to provide a multiresolution sequence of polynomial approximations to functions on the sphere. The needlet construction makes use of 
  integration rules that are exact for polynomials up to a given degree. The aim of the present paper is to relax the exactness 
  of the integration rules by replacing them with QMC designs as introduced by Brauchart, Saff, Sloan, and Womersley (2014).
  Such integration rules
  (generalised here by allowing non-equal cubature weights) provide the same asymptotic order of convergence as exact rules for Sobolev
  spaces $\mathbb{H}^s$, but are easier to obtain numerically. With such rules we construct ``generalised needlets''. The paper
  provides an error analysis that allows the replacement of the original needlets by generalised needlets, and more generally, analyses a hybrid scheme in which the needlets for the lower levels are of the traditional kind, whereas the new generalised needlets are used for some number of higher levels.  Numerical experiments complete the paper.
\end{abstract}


\maketitle

\section{Introduction}
\label{sec:intro}

Spherical needlets, introduced by F. J. Narcowich, P. Petrushev, and
J. D. Ward~\cite{NaPeWa2006-1,NaPeWa2006-2} provide an elegant multiresolution
sequence of polynomial approximations on the sphere. The level $J$
approximation of a square-integrable function $f$ has the form
\begin{equation}\label{eq:VJ}
V_J^{{\rm need}}(f;\bsx) := \sum_{j=0}^J\sum_{k=1}^{n_j}\left(f,\psi_{j,k}\right)\psi_{j,k}(\bsx),
\qquad \bsx\in \mathbb{S}^d,
\end{equation}
where the inner product is the $\mathbb{L}_2$ inner product
\begin{equation}\label{eq:inner}
(f,g) := \int_{\mathbb{S}^d} f(\bsx) \, g(\bsx) \, \rd \sigma(\bsx),
\end{equation}
$\mathbb{S}^d$ is the unit sphere in $\bbR^{d+1}$, and $\sigma$ is the
normalised surface area measure on $\mathbb{S}^d$.
The needlet $\psi_{j,k}$ is a spherical polynomial of degree $2^j-1$,
which is band limited, in that it is orthogonal to all spherical
polynomials of degree at most $2^{j-2}$. The second label $k$ in
$\psi_{j,k}$ indicates its location on the sphere, in the sense that each
collection $\{\psi_{j,k}: k = 1,\ldots, n_j\}$ of needlets for level $j$ for
$j=0, 1, 2, \ldots$ is associated with a cubature rule on the sphere (the
``needlet cubature rule'') with $n_j$ points $\bsx_{j,k}, k = 1, \ldots, n_j$
and positive weights $w_{j,k}, k = 1,\ldots,n_j$, where the cubature rule is
required to be exact for all polynomials up to degree $2^{j+1}-1$; i.e.,
\begin{equation}\label{eq:polyprec}
\sum_{k=1}^{n_j}w_{j,k} \, p(\bsx_{j,k}) = \int_{\mathbb{S}^d}  p(\bsx) \, \rd
\sigma(\bsx), \quad p \in \mathbb{P}_{2^{j+1}-1}, \quad j=0, 1, 2, \ldots.
\end{equation}
%
Here, $\mathbb{P}_\ell$ is the set of all spherical polynomials on
$\mathbb{S}^d$ of degree less than or equal to $\ell$. The needlet
$\psi_{j,k}$ is a zonal function centered at the cubature point $\bsx_{j,k}$
in the sense that $\psi_{j,k}(\bsx)$ depends only on the geodesic distance of $\bsx$ from $\bsx_{j,k}$, or equivalently, on the Euclidean inner product $\IP{\bsx}{\bsx_{j,k}}$.
%
%
Precisely, the needlet $\psi_{j,k}$ is defined by
\begin{equation}\label{eq:need}
\psi_{0,k}(\bsx) := \sqrt{w_{0,k}}, \quad
\psi_{j,k}(\bsx) := \sqrt{w_{j,k}} \sum_{\ell=0}^\infty \ffilter{\frac{\ell}{2^{j-1}}} Z(d,\ell)
P_\ell( \IP{ \bsx_{j,k} }{ \bsx }), \quad j\ge 1, \,\bsx\in \mathbb{S}^d,
\end{equation}
where $h$ (a ``band limited'' filter) is a non-negative $C^\kappa$
function on $\mathbb{R}_+$, $\kappa\in\mathbb{N}$,  with support
$[1/2,2]$ (making the sum finite) and satisfying
\begin{equation}\label{eq:hprop}
  \filter{t}^2 + \filter{2t}^2 = 1, \qquad \text{$t\in[1/2,1]$},
\end{equation}
from which it follows that $h(1)= 1$, while $h(1/2)= h(2) = 0$.
Here, $Z(d, \ell)$ is the dimension of the space of the homogeneous harmonic polynomials
of exact degree $\ell$ on $\mathbb{S}^d$ and $P_\ell$ is the
(generalised) Legendre polynomial associated with $\mathbb{S}^d$ normalised so that $P_\ell(1) = 1$.
Other properties of needlets are given in Section \ref{sec:needletdef}.

Our aim in this paper is to liberate needlets from this exactness
requirement of the integration rule for polynomials of potentially high degree, because of its
very restrictive nature.  Instead, we will define ``generalised needlets''
$\Psi_{j,k}$, in a way that preserves the multiresolution form of the
approximation, in analogy with \eqref{eq:VJ}.  These generalised needlets
too will be associated with a cubature rule (a ``generalised needlet
cubature rule''), but now without the polynomial precision requirement
spelled out in \eqref{eq:polyprec}. Instead, we will require only that the
set of generalised needlet cubature rules yields the
optimal rate of convergence for functions with a certain smoothness, see
\eqref{eq:QMCprop}. To avoid confusion we will use capital letters to
denote the points, weights and number of points in the generalised needlet
cubature rules. Thus the generalised needlet approximation will take the
form, in complete analogy to \eqref{eq:VJ},
\begin{equation}\label{eq:VJgen}
V_J^{{\rm gen}}(f;\bsx) :=
\sum_{j=0}^J\sum_{k=1}^{N_j}\left(f,\Psi_{j,k}\right)\Psi_{j,k}(\bsx),
\qquad \bsx\in \mathbb{S}^d,
\end{equation}
with
\begin{equation}\label{eq:gen_need}
\Psi_{0,k}(\bsx) := \sqrt{W_{0,k}},  \quad \Psi_{j,k}(\bsx) := \sqrt{W_{j,k}} \sum_{\ell=0}^\infty \ffilter{\frac{\ell}{2^{j-1}}} Z(d,\ell)
P_\ell(\IP{ \bsX_{j,k} }{ \bsx }), \; j\ge 1, \; \bsx\in \mathbb{S}^d,
\end{equation}
so that $\Psi_{j,k}$ is another zonal polynomial of degree $2^j-1$, this one centered at the
cubature point~$\bsX_{j,k}$.

To construct the generalised needlet cubature rules, we assume we are
given an infinite sequence of $N$-point cubature rules on $\mathbb{S}^d$,
with the $N$-point rule having points and weights labelled $X_1^{(N)},
X_2^{(N)},\ldots, X_N^{(N)}$ and  $W_1^{(N)}, W_2^{(N)},\ldots, W_N^{(N)}$
respectively, and such that
\begin {equation}\label{eq:QMCprop}
\sup_{\substack{f\in \mathbb{H}^s(\mathbb{S}^d), \\
\|f\|_{\mathbb{H}^s}\le 1}} \left| \sum_{k=1}^{N} W_{k}^{(N)} \,
f(\bsX_{k}^{(N)}) - \int_{\mathbb{S}^d} f(\bsx) \, \rd \sigma(\bsx)
\right| \le \frac{\Cqmc}{N^{s/d}},
\end{equation}
for some positive constant $\Cqmc$ that may depend on $d$, $s$, $h$ and the
particular sequence of cubature rules, but does not depend on $N$. Here
$\mathbb{H}^s(\mathbb{S}^d)$ is the Sobolev space of smoothness $s\ge 0$,
to be defined in Section \ref{sec:prelim}.  This space may be thought of
as the space of functions in $\mathbb{L}_2(\mathbb{S}^d)$ with $s$
square-integrable derivatives.  The maximal value $s^*$ of $s$ for the
given sequence of cubature rules is called the \emph{strength} of the
particular sequence.

Given the sequence of cubature rules satisfying \eqref{eq:QMCprop}, we
select an increasing subsequence $N_{i_1}, N_{i_2} \ldots, N_{i_J}$ of length $J$.  Without loss of generality we relabel $N_{i_j}$ as $N_j$, and
define the $j$th generalised needlet cubature rule to be the member of the
above sequence with $N_j = N_{i_j}$, and label the points and weights as
$X_{j,k} = X_k^{(N_j)}$ and $W_{j,k} = W_k^{(N_j)}$ respectively.  Thus the
$j$th generalised cubature rule satisfies
\begin{equation}
\label{eq:QMCpropj}
\sup_{\substack{f\in \mathbb{H}^s(\mathbb{S}^d), \\ \|f\|_{\mathbb{H}^s}\le 1}} \left| \sum_{k=1}^{N_j} W_{j,k} \, f(\bsX_{j,k}) -
\int_{\mathbb{S}^d} f(\bsx) \, \rd \sigma(\bsx) \right|
\le \frac{\Cqmc}{N_j^{s/d}},  \qquad  s \in (d/2,s^*).
\end{equation}

The construction of the generalised needlet cubature rules as described
here makes it clear that the number of points $N_j$ for the $j$th such
rule is under the control of the user, with the recognition that higher
numbers of points $N_j$ will be paid for by higher cost.

Sequences of cubature rules with the property~\eqref{eq:QMCprop} certainly
exist, firstly because all positive-weight rules that satisfy
\eqref{eq:polyprec} and for which the number of points $N_j$ satisfies
$N_j\asymp 2^{jd}$ (meaning that there exist $c_1, c_2 >0$ such that $c_1
2^{jd} \le N_j \le c_2 2^{jd}$ for all~$j$) automatically satisfy the
condition \eqref{eq:QMCprop}, see \cite{BrHe2007}. Secondly, such rules
generalise in a natural way so-called ``QMC designs'', introduced in the paper \cite{BrSaSlWo2014}: QMC designs are the special case of our
new rules in which all of the weights $W_{j,k}$ are equal, and equal to
$1/N_j$. That paper proves the existence of some rules that satisfy
\eqref{eq:QMCprop} yet not property~\eqref{eq:polyprec}.
One example is provided by $N$-point sets that maximise the sum of all
mutual Euclidean distances. Such point sets are proved in
\cite{BrSaSlWo2014} to form a sequence of QMC designs
of strength at least $ (d+1)/2$.
(Numerical experiments in \cite{BrSaSlWo2014} would suggest that $s^*=4$ is achieved for $d=2$.)
That paper also demonstrated empirically the likelihood that many other
well known point sets on the sphere are QMC designs for moderate values of~$s^*$. It is a reasonable expectation that many sequences of rules with the property
\eqref{eq:QMCprop} exist for usefully large values of~$s^*$, especially as
we are here allowing much greater freedom in the choice of weights.


Our aim is to show that the generalised needlet approximation has
properties comparable to \eqref{eq:VJ}, if the cubature rules are well
chosen.

In this paper we focus on $\mathbb{L}_2(\mathbb{S}^d)$ error estimates.
For the classical needlet construction it is known that for a
sufficiently smooth kernel $h$ and for $f \in H^s(\mathbb{S}^d)$ the $\mathbb{L}_2$ error satisfies
\begin{equation}\label{eq:need_error}
\|V_J^{{\rm need}}(f; \cdot )-f\|_{\mathbb{L}_2} \le \frac{C'}{2^{J\,s}} \, \|f\|_{\mathbb{H}^s},
\end{equation}
which is the optimal convergence order for a polynomial approximation of the given degree.  The constant $C' > 0$ may depend on $d,s$ and $h$, but does not depend on $J$.  

For the generalised needlet construction of this paper it turns out to be
advantageous for the smaller values of $j$ to use traditional needlet
cubature rules, say for $j=0$ up to $J_0$. We shall refer to this construction as
the ``hybrid'' implementation. The $\mathbb{L}_2$ error
is then bounded, as we show in Section \ref{sec:hybrid}, by
\begin{equation} \label{eq:gen_error}
\|V_J^{{\rm gen}}(f; \cdot )-f\|_{\mathbb{L}_2} \leq \frac{C'}{2^{J\,s}} \, \|f\|_{\mathbb{H}^s} \,
 + \, \frac{C^{''}}{2^{(J_0-1)s}} \left( \sum_{j=J_0+1}^J \frac{2^{j(s+d)}}{N_j^{s/d}} \right)
\|f\|_{\mathbb{H}^s},
\end{equation}
in which the first term is the original needlet error
\eqref{eq:need_error}, and the second term is a correction due to the use
of the generalised needlet cubature rule for $j > J_0$. Obviously, if the generalised
needlet approximation is to improve upon the true needlet approximation at
level $J_0$  then according to \eqref{eq:gen_error}
the sum in parentheses must be suitably small.  This can always be achieved by choosing the $N_j$ for $j > J_0$ to be sufficiently large. 

The underlying reason why traditional cubature is desirable for the
smaller values of~$j$, and hence the justification for the hybrid
implementation, lies in the extraordinary efficiency of traditional
cubature rules for polynomials of low degree: for example,
exact integration of a polynomial of degree $\ell$ on $\mathbb{S}^2$  by a
product rule requires only $\frac{\ell^{2}}{2} + \mathcal{O}(\ell)$ points
(see \cite{HeSlWo2015}).
In contrast, by their very definition QMC design rules are effective
only for larger values of $N$.
On the other hand, the freedom to use a much wider range of cubature
rules gives great flexibility. The greater freedom of choice is
of special value if equal-weight cubature rules are preferred in the
needlet design, because when $N$ is large spherical designs are hard to construct and known rules are in very short supply
(although for $\mathbb{S}^2$ computed efficient symmetric spherical designs with good geometric properties are available~\cite{Womersley2018}
for degrees up to $325$ with $52,978$ points).
On the other hand, for example,
for $\mathbb{S}^2$ the Bauer generalised spiral points~\cite{Bauer2000,RaSaZh1994}
appear to be QMC designs with $s^*=3$, see \cite{BrSaSlWo2014},
yet can be comfortably evaluated for millions of points.
Effective construction of QMC design sequences is a topic of
ongoing research.


The paper is organised as follows.
In Section~\ref{sec:prelim} we collect useful facts regarding spherical
harmonics, our Sobolev space setting, which is considered as a reproducing
kernel Hilbert space, error of cubature rules, and QMC design sequences.
Section~\ref{sec:needletdef} gives details on filters and needlets.
Section \ref{sec:hybrid} describes the hybrid implementation, in
which traditional needlets are used for $j\le J_0$, and derives $\mathbb{L}_2$
error bounds.  The main result is Theorem \ref{thm:main}. The final section contains numerical experiments to test the theory.

\section{Preliminaries}
\label{sec:prelim}

We are concerned with real-valued functions defined on the unit
sphere $\mathbb{S}^d := \{ \bsx \in \bbR^{d+1} \, : \, \IP{ \bsx }{ \bsx }
= 1 \}$ in the Euclidean space $\bbR^{d+1}$, $d \geq 2$, provided with the
inner product \eqref{eq:inner} and induced norm $\| \cdot \|_2:=
(\cdot,\cdot)^{1/2}$. The measure on $\mathbb{S}^d$ is the normalised
Lebesgue surface area measure $\sigma$, whereas~$\omega_d$ denotes the
non-normalised surface area of $\mathbb{S}^d$ and is given by $\omega_0 =
2$ and
\begin{equation} \label{eq:omega.d.ratio}
\frac{\omega_{d-1}}{\omega_d}
 = \frac{\Gamma( (d + 1)/2 )}{\sqrt{\pi} \, \Gamma( d/2 )},
 \qquad \frac{\omega_{d-1}}{\omega_d} \int_{-1}^1 \left( 1 - t^2 \right)^{d/2-1} \rd t = 1;
\end{equation}
here, $\Gamma( z )$ denotes the gamma function.

We shall use the Pochhammer symbol (rising factorial) defined recursively by
\begin{equation*}
( a )_0 := 1, \qquad ( a )_{n} := ( a )_{n-1} ( n - 1 + a ), \quad n \geq 1,
\end{equation*}
and given explicitly by
\begin{equation*}
( a )_n = \frac{\Gamma( n + a )}{\Gamma( a )},
\end{equation*}
whenever the right-hand side is defined.

\subsection{Spherical harmonics}

The restriction to $\mathbb{S}^d$ of a homogeneous and harmonic polynomial of total degree $\ell$ defined on $\bbR^{d+1}$ is called a \emph{spherical harmonic} of degree $\ell$ on~$\mathbb{S}^d$. The family $\mathcal{H}_{\ell}^{d} = \mathcal{H}_{\ell}^{d}( \mathbb{S}^d )$ of all spherical harmonics of exact degree $\ell$ on $\mathbb{S}^d$ has dimension
\begin{equation} \label{eq:Z.d.ell}
Z(d,\ell) := \left( 2\ell + d - 1 \right) \frac{\Gamma( \ell + d - 1 )}{\Gamma( d ) \Gamma( \ell + 1 )} = \frac{2 \ell + d - 1}{d - 1} \, \frac{(d-1)_\ell}{\ell!} \sim \frac{2}{\Gamma(d)} \, \ell^{d-1} \quad \text{as $\ell \to \infty$.}
\end{equation}
Each spherical harmonic $Y_{\ell}$ of exact degree $\ell$ is an eigenfunction of the negative \emph{Laplace-Beltrami operator} $-\LBo{d}$ for $\mathbb{S}^d$, with eigenvalue
\begin{equation} \label{eq:eigenvalue}
\lambda_\ell := \ell \left( \ell + d - 1 \right), \qquad \ell = 0, 1, 2, \ldots.
\end{equation}
%
We shall denote by $\mathbb{P}_\ell = \mathbb{P}_\ell( \mathbb{S}^d )$ the space of all spherical polynomials of degree $\leq \ell$. This space coincides with the span of all spherical harmonics up to (and including) degree $\ell$, and its dimension is 
\begin{equation} \label{eq:dim.formula}
\mathrm{dim}( \mathbb{P}_\ell ) = \sum_{\ell^\prime=0}^\ell Z(d,\ell^\prime) = Z(d+1,\ell).
\end{equation}

As usual, we let $\{ Y_{\ell, m} : m = 1, \ldots, Z(d, \ell) \}$ denote a
real $\mathbb{L}_2$-orthonormal basis of $\mathcal{H}_{\ell}^{d}$. Then
the basis functions $Y_{\ell, k}$ satisfy the identity known as the
\emph{addition theorem}:
\begin{equation} \label{eq:addition.theorem}
\sum_{m=1}^{Z(d,\ell)} Y_{\ell,m}( \bsx ) Y_{\ell,m}( \bsy ) = Z(d,\ell) \, P_\ell( \IP{ \bsx }{ \bsy } ), \qquad \bsx, \bsy \in \mathbb{S}^d.
\end{equation}
Here, $P_\ell$ is the generalised  Legendre polynomial, orthogonal on the interval $[-1,1]$ with respect to the weight function $(1-t^2)^{d/2-1}$, and normalised by $P_\ell(1) = 1$. Notice that for $d \geq 2$ and $\lambda = \frac{d-1}{2}$,
\begin{equation*}
Z(d,\ell) P_\ell(x)
=
\frac{\ell+\lambda}{\lambda} \, C_\ell^\lambda(x)
=
\frac{\ell+\lambda}{\lambda} \, \frac{(2\lambda)_\ell}{( \lambda + 1/2 )_\ell} \, P_\ell^{(\lambda-1/2,\lambda-1/2)}(x),
\end{equation*}
where $C_n^\lambda$ is the $n$-th Gegenbauer polynomial with index $\lambda$ and $P_n^{(\alpha,\beta)}$ is the $n$-th Jacobi polynomial with indices $\alpha$ and $\beta$ (see \cite{NIST:DLMF}).

The collection $\{ Y_{\ell, m} : m = 1, \ldots, Z(d, \ell); \ell = 0, 1, \ldots \}$ forms a complete
orthonormal (with respect to $\sigma$) system for the Hilbert space $\mathbb{L}_2(\mathbb{S}^d)$ of
square-integrable functions on~$\mathbb{S}^d$ endowed with the inner product
\eqref{eq:inner} 
%
(For more details, we refer the reader to \cite{BeBuPa1968, Mu1966}.)

The Funk-Hecke formula states that for every spherical harmonic $Y_\ell$ of degree $\ell$ (see~\cite{Mu1966}),
\begin{equation}
\label{eq:Funk-Hecke.formula}
\int_{\mathbb{S}^d} g( \IP{ \bsy}{ \bsz } ) \, Y_\ell( \bsy ) \, \rd \sigma( \bsy ) = \widehat{g}(\ell) \, Y_\ell( \bsz ), \qquad \bsz \in \mathbb{S}^d,
\end{equation}
where
\begin{equation} \label{eq:Funk-Hecke.formula.B}
\widehat{g}( \ell ) = \frac{\omega_{d-1}}{\omega_d} \, \int_{-1}^1 g( t ) \, P_\ell( t ) \left( 1 - t^2 \right)^{d/2-1} \rd t.
\end{equation}
(This formula holds, in particular, for the spherical harmonic ${Y_\ell(
\bsy ):= P_\ell( \IP{ \bsa}{ \bsy } )}$, $\bsa \in \mathbb{S}^d$.)

\subsection{Sobolev spaces}

The Sobolev space $\mathbb{H}^s(\mathbb{S}^d)$ may be defined for $s\ge 0$ as the set of all functions $f\in \mathbb{L}_2(\mathbb{S}^d)$ whose Laplace-Fourier coefficients
\begin{equation} \label{eq:L-F.coeff}
\widehat{f}_{\ell,m} := ( f, Y_{\ell,m} ) = \int_{\mathbb{S}^d} f( \bsx ) \, Y_{\ell,m}( \bsx ) \, \rd \sigma( \bsx )
\end{equation}
satisfy
\begin{equation}\label{eq:sobcond}
\sum_{\ell=0}^\infty \left( 1 + \lambda_\ell \right)^{s} \sum_{m=1}^{Z(d,\ell)} \left| \widehat{f}_{\ell,m} \right|^2 <\infty,
\end{equation}
where the $\lambda_\ell$'s are given in \eqref{eq:eigenvalue}.
On setting $s=0$, we recover $\mathbb{H}^0(\mathbb{S}^d)=\mathbb{L}_2(\mathbb{S}^d).$

Let $s > d/2$ be fixed and suppose we are given a sequence of positive real numbers $(a_\ell^{(s)})_{\ell \geq 0}$  satisfying \footnote{We write $a_n\asymp b_n$ to mean that there exist positive constants $c_1$ and $ c_2$ independent of $n$ such that $c_1 a_n\le b_n\le c_2 a_n$ for all $n$.}
\begin{equation} \label{eq:sequenceassumption}
a_\ell^{(s)} \asymp \left( 1 + \lambda_\ell \right)^{-s} \asymp \left( 1 + \ell \right)^{-2s}.
\end{equation}
Then we can define a norm on $\mathbb{H}^s(\mathbb{S}^d)$ by
\begin{equation}\label{eq:sobnorm}
\|f\|_{\mathbb{H}^s} := \left[\sum_{\ell=0}^\infty \frac{1}{a_\ell^{(s)}} \sum_{m=1}^{Z(d,\ell)} \left|\widehat{f}_{\ell,m}\right|^2\right]^{1/2}.
\end{equation}
The norm therefore depends on the particular choice of the sequence $(a_\ell^{(s)})_{\ell \geq 0}$, but for notational simplicity we shall generally not show this dependence explicitly. Notice that the space $\mathbb{H}^s$ only depends on $s$; norms for different choices of the sequence $(a_\ell^{(s)})_{\ell\geq0}$ satisfying  \eqref{eq:sequenceassumption} are equivalent.
The corresponding inner product in the Sobolev space is
\begin{equation}\label{inner}
( f, g )_{\mathbb{H}^s} := \sum_{\ell=0}^\infty \frac{1}{a_\ell^{(s)}} \sum_{m=1}^{Z(d,\ell)} \widehat{f}_{\ell,m} \, \widehat{g}_{\ell,m}.
\end{equation}

It is well known that $\mathbb{H}^{s}(\mathbb{S}^d) \subset \mathbb{H}^{s^\prime}(\mathbb{S}^d)$ whenever $s > s^\prime$, and that $\mathbb{H}^{s}(\mathbb{S}^d)$ is embedded in the space of $k$-times continuously differentiable functions $C^k(\mathbb{S}^d)$ if $s > k + d / 2$ (see e.g., \cite{He2006}).

\subsection{Sobolev spaces as reproducing kernel Hilbert spaces}
\label{subsec:H.as.RKHS}

As the point-evaluation functional is bounded in $\mathbb{H}^s(\mathbb{S}^d)$ whenever $s>d/2$, the Riesz representation theorem assures the existence of a unique reproducing kernel $K^{(s)}$, which can be written as
\begin{equation}\label{eq:K}
K^{(s)}(\bsx,\bsy) = \sum_{\ell=0}^\infty a_\ell^{(s)} Z(d,\ell) P_\ell(\IP{ \bsx}{ \bsy }) = \sum_{\ell=0}^\infty a_\ell^{(s)} \sum_{m=1}^{Z(d,\ell)} Y_{\ell,m}(\bsx) Y_{\ell,m}(\bsy),
\end{equation}
where the positive coefficients $a_\ell^{(s)}$ satisfy \eqref{eq:sequenceassumption}.
It is easily verified that the above expression has the reproducing kernel properties
\begin{equation} \label{eq:repr.kernel.prop}
K^{(s)}( \cdot, \bsx ) \in \mathbb{H}^s(\mathbb{S}^d), \quad \bsx \in \mathbb{S}^d, \qquad ( f, K^{(s)}(\cdot, \bsx) )_{\mathbb{H}^s} = f(\bsx), \quad \bsx \in \mathbb{S}^d, f \in \mathbb{H}^s(\mathbb{S}^d).
\end{equation}
The kernel is a {\em zonal} function; i.e., $K^{(s)}(\bsx,\bsy)$ depends only on the inner product $\IP{ \bsx}{ \bsy }$. Observe that
\begin{equation*}
\left\| \bsx - \bsy \right\|_2^2 = 2 - 2 \IP{\bsx}{\bsy}, \qquad \bsx, \bsy \in \mathbb{S}^d.
\end{equation*}

We next give several examples of reproducing kernels studied in the literature. It was observed in \cite{SlWo2004} that the Cui and Freeden kernel (see \cite{CuFr1997}) on $\mathbb{S}^2$,
\begin{equation*} 
K_\mathrm{CF}(\bsx, \bsy):= 2 - 2 \log\Big( 1 + \frac{1}{2} \left\| \bsx - \bsy \right\|_2 \Big) = 1 + \sum_{\ell=1}^\infty \frac{1}{\ell ( \ell + 1 )} \, P_\ell( \IP{ \bsx}{ \bsy } ), \quad \bsx,\bsy \in \mathbb{S}^2,
\end{equation*}
is a reproducing kernel for $\mathbb{H}^{3/2}( \mathbb{S}^2 )$. The distance kernel
\begin{equation*}
K_{\mathrm{dist}}(\bsx, \bsy) := \frac{8}{3} - \left\| \bsx - \bsy \right\|_2 = \frac{4}{3} + \sum_{\ell=1}^\infty \frac{1}{\left( \ell + \frac{3}{2} \right) \left( \ell - \frac{1}{2} \right)} \, P_\ell( \IP{ \bsx}{ \bsy } ), \quad \bsx,\bsy \in \mathbb{S}^2,
\end{equation*}
considered in \cite{BrDi2013} is also a reproducing kernel for this space for a different but equivalent norm. These properties can be seen from the Fourier-Laplace expansions. It can be further shown that the generalised distance kernel
\begin{equation*}
K_{\mathrm{gdist}}(\bsx, \bsy) := 2 V_{d-2s}( \mathbb{S}^d ) - \left\| \bsx - \bsy \right\|_2^{2s-d}, \qquad \bsx,\bsy \in \mathbb{S}^d,
\end{equation*}
where
\begin{equation*}
V_{d-2s}( \mathbb{S}^d ) := \int_{\mathbb{S}^d} \int_{\mathbb{S}^d} \left\| \bsx - \bsy \right\|_2^{2s-d} \rd \sigma( \bsx ) \rd \sigma( \bsy ) = 2^{2s-1} \frac{\Gamma((d+1)/2) \Gamma(s)}{\sqrt{\pi} \Gamma(d/2+s)},
\end{equation*}
is a reproducing kernel for $\mathbb{H}^{s}( \mathbb{S}^d )$ for $d/2 < s < (d + 1) / 2$; cf. \cite{BrDi2013} and see \cite{BrDi2013b} for distance kernels for smoother Sobolev spaces over $\mathbb{S}^d$.

\subsection{Convergence in Sobolev spaces for cubature rules with polynomial precision}

The {\em worst-case error} in the Sobolev space
$\mathbb{H}^s(\mathbb{S}^d)$ for the positive weight cubature rule
\begin{equation*}
\mathrm{Q}[\mathcal{X}_N] := \sum_{k=1}^N w_k \, f(\bsx_k),
\end{equation*}
with weight and node set
$\mathcal{X}_N = \{ (w_k, \bsx_k ) : k = 1,\ldots, N \}
\subset \mathbb{R}_+ \times \mathbb{S}^d$,
approximating the integral
\begin{equation*} 
\mathrm{I}(f) := \int_{\mathbb{S}^d} f( \bsx ) \rd \sigma( \bsx ),
\end{equation*}
is given by
\begin{equation*} 
\wce(\mathrm{Q}[\mathcal{X}_N]; \mathbb{H}^{s}( \mathbb{S}^d ) ) := \sup \left\{ \big| \mathrm{Q}[\mathcal{X}_N](f) - \mathrm{I}(f) \big| : f \in \mathbb{H}^{s}( \mathbb{S}^d ), \| f \|_{\mathbb{H}^s} \leq 1 \right\}.
\end{equation*}
The Koksma-Hlawka type error bound
\begin{equation*} 
\left| \mathrm{Q}[\mathcal{X}_N](f) - \mathrm{I}(f) \right| \leq \wce(\mathrm{Q}[\mathcal{X}_N]; \mathbb{H}^{s}( \mathbb{S}^d ) ) \left\| f \right\|_{\mathbb{H}^s}
\end{equation*}
for an arbitrary function ${f \in \mathbb{H}^{s}( \mathbb{S}^d )}$ follows from this definition.
The reproducing kernel property $f(\bsx)=( f, K^{(s)}(\cdot,\bsx))_{\mathbb{H}^s}$ allows us to write
\begin{equation*}
\mathrm{Q}[\mathcal{X}_N](f) - \mathrm{I}(f) = ( f, \mathcal{R}[X_N] )_{\mathbb{H}^s}, \qquad f \in \mathbb{H}^{s}( \mathbb{S}^d ),
\end{equation*}
where $\mathcal{R}[\mathcal{X}_N] \in \mathbb{H}^{s}( \mathbb{S}^d )$ is the ``representer'' of the error, given by
\begin{equation*}
\mathcal{R}[\mathcal{X}_N]( \bsx ) := \sum_{j = 1}^{N} w_j \, K^{(s)}(\bsx, \bsx_j) - \mathrm{I}_{\bsy} K^{(s)}(\bsx,\cdot).
\end{equation*}
Here, $\mathrm{I}_{\bsy}
K^{(s)}$ means the integral functional $\mathrm{I}$ applied to the second
variable in $K^{(s)}$. A standard argument yields
\begin{equation*}
\left[ \wce(\mathrm{Q}[\mathcal{X}_N]; \mathbb{H}^{s}( \mathbb{S}^d ) ) \right]^2 = \sum_{j=1}^{N} \sum_{k=1}^{N} w_j w_k \, K^{(s)}(\bsx_{j},\bsx_{k}) - 2 \sum_{j=1}^{N} w_j \, \mathrm{I}_{\bsy} K^{(s)}(\bsx_{j},\cdot) + \mathrm{I}_{\bsx} I_{\bsy} K^{(s)}.
\end{equation*}
It follows from \eqref{eq:K} and $\mathrm{I}(1) = 1$ that
\begin{equation*}
\mathrm{I}_{\bsy} K^{(s)}(\bsx,\cdot) = a_0^{(s)},
\end{equation*}
from which we deduce
\begin{equation} \label{eq:wce2}
\begin{split}
\left[ \wce(\mathrm{Q}[\mathcal{X}_N]; \mathbb{H}^{s}( \mathbb{S}^d ) ) \right]^2
&= \left[ \sum_{j=1}^{N} \sum_{k=1}^{N} w_j w_k \, K^{(s)}(\bsx_{j},\bsx_{k}) \right] - \left( 2 \sum_{j=1}^N w_j - 1 \right) a_0^{(s)} \\
&= \sum_{j=1}^{N} \sum_{k=1}^{N} w_j w_k \, \mathcal{K}^{(s)}(\IP{ \bsx}{ \bsy }) + \left( \sum_{j=1}^N w_j - 1 \right)^2 a_0^{(s)},
\end{split}
\end{equation}
where $\mathcal{K}^{(s)}:[-1,1] \to \mathbb{R}$ is defined by
\begin{equation} \label{eq:calK}
\mathcal{K}^{(s)}(t) := \sum_{\ell=1}^\infty a_\ell^{(s)} Z(d,\ell) P_\ell(t),\quad t \in [-1,1]. 
\end{equation}
The use of the calligraphic symbol $\mathcal{K}$ here and for
subsequent kernels indicates that the sum over $\ell$ runs from
$\ell=1$ rather than $\ell=0$. If the weights satisfy
\begin{equation*}
\sum_{j=1}^N w_j = 1,
\end{equation*}
then constant functions (spherical polynomials of degree $0$) are
integrated exactly. For such cubature rules $Q[\mathcal{X}_N]$ we get
\begin{equation*}
\left[ \wce(\mathrm{Q}[\mathcal{X}_N]; \mathbb{H}^{s}( \mathbb{S}^d ) ) \right]^2 = \sum_{j=1}^{N} \sum_{k=1}^{N} w_j w_k \, \mathcal{K}^{(s)}(\IP{ \bsx}{ \bsy }).
\end{equation*}

A positive weight cubature rule is exact for all spherical polynomials of degree $\leq t$ if
\begin{equation*}
\mathrm{Q}[\mathcal{X}_N]( p ) = \mathrm{I}(p) \qquad \text{for all $p \in \mathbb{P}_t( \mathbb{S}^d )$.}
\end{equation*}
Sequences of positive weight cubature rules with polynomial precision have
a known fast-convergence property in Sobolev spaces, stated in the
following theorem. This property was first proved for the particular case
$s=3/2$ and $d=2$ in \cite{HeSl2005}, then extended to all $s>1$ for $d=2$
in \cite{HeSl2006}, and finally extended to all $s>d/2$ and all $d\ge 2$
in \cite{BrHe2007}.

\begin{theorem}
Given $s > d/2$, there exists $C(d,s) > 0$ depending on the $\mathbb{H}^s(
\mathbb{S}^d )$-norm such that for every $N$-point positive weight
cubature rule $\mathrm{Q}[\mathcal{X}_N]$ on $\mathbb{S}^d$ which is exact
for all spherical polynomials of degree $\leq t$ there holds
\begin{equation} \label{eq:asymptotic_in_t.bound}
\wce(\mathrm{Q}[\mathcal{X}_N]; \mathbb{H}^{s}( \mathbb{S}^d ) ) \leq \frac{C(d,s)}{t^s}.
\end{equation}
\end{theorem}

The relation between $N$ and $t$ is not fixed. It is known that there
exist positive weight tensor product cubature rules with order $t^d$
points, see \cite{HeSlWo2015}, and also spherical $t$-designs with
order-optimal number of points \cite{BoRaVi2013}. All such rules
provide upper bounds of order $N^{-s/d}$ that match the
following optimal lower bounds for the worst-case error (see
\cite{He2006} and \cite{HeSl2005b}).

\begin{theorem}
Given $s > d/2$, there exists $c(d,s) > 0$ depending on the $\mathbb{H}^s(
\mathbb{S}^d )$-norm such that for any $N$ points $\bsx_1, \dots, \bsx_N
\in \mathbb{S}^d$ and any associated weights $w_1, \dots, w_N \in
\mathbb{R}$ the corresponding cubature rule $\mathrm{Q}_N(f) :=
\sum_{j=1}^N w_j f( \bsx_j )$ satisfies
\begin{equation*}
\wce(\mathrm{Q}_N; \mathbb{H}^{s}( \mathbb{S}^d ) ) \geq \frac{c(d,s)}{N^{s/d}}.
\end{equation*}
\end{theorem}

\subsection{QMC design sequences}

Motivated by the results of the previous subsection, the paper \cite{BrSaSlWo2014} introduced the following concept for equal-weight cubature.
\begin{definition} \label{def:specific_approx.sph.design}
Given $s > d/2$, a sequence $(X_N)$ of $N$-point configurations  on
$\mathbb{S}^d$ with $N\to\infty$ is said to be a {\em sequence of QMC
designs for $\mathbb{H}^s( \mathbb{S}^d )$} if there exists
$\Cqmc>0$, independent of $N$, such that
\begin{equation}\label{eq:approxfors}
\sup_{\substack{f \in \mathbb{H}^s( \mathbb{S}^d ), \\
\| f \|_{\mathbb{H}^s} \leq 1}} \Bigg| \frac{1}{N} \sum_{\bsx \in X_N}
f( \bsx ) - \int_{\mathbb{S}^d} f( \bsx ) \rd \sigma_d( \bsx ) \Bigg|  \leq \frac{\Cqmc}{N^{s/d}}.
\end{equation}
\end{definition}
In this definition $X_N$ need not be defined for all natural numbers $N$: it is sufficient that $X_N$ exists for an infinite subset of the natural numbers.

It is known \cite[Thm.~14]{BrSaSlWo2014}, for all $N\ge 1$,  that
points $\bsx_1,\ldots,\bsx_N \in \mathbb{S}^d$ maximising the sum of
generalised Euclidean distances,
$\sum_{j,k=1}^{N} \left|\mathbf{x}_j - \mathbf{x}_k \right|^{2 s-d}$,
form a QMC design sequence $(X_{N,s}^*)_{N\in \mathbb{N}}$ for
$\mathbb{H}^{s}(\mathbb{S}^d )$ if $s \in (d/2, d/2 + 1 )$.
It is also known that a sequence $( Z_{N_t} )_{t \in \mathbb{N}}$ of
spherical $t$-designs with the optimal order of points, $N_t \asymp
t^d$, has the remarkable property \cite[Thm.~6]{BrSaSlWo2014} that \eqref{eq:approxfors} holds for \emph{all} $s>d/2$. Thus a spherical design sequence is a QMC design sequence for
$\mathbb{H}^{s}( \mathbb{S}^d )$ for every $s > d/2$.
\section{Needlet definition}
\label{sec:needletdef}

\subsection{Filters}

Set $\mathbb{R}_+ := [0, \infty)$. A continuous compactly supported
function $h:\mathbb{R}_+ \to \mathbb{R}_+$ is said to be a filter. We
shall only consider filters with support $[1/2,2]$; i.e.,
following~\cite{WaLGSlWo2017} (also, cf.
\cite{NaPeWa2006-1,NaPeWa2006-2}), we shall assume that $h$ is a
filter satisfying
\begin{equation}
h \in C^\kappa( \mathbb{R}_+ ), \qquad \kappa \in \mathbb{N}, \label{eq:filter.prop.1}
\end{equation}
\begin{equation}
\supp h = \big[ 1/2, 2 \big], \label{eq:filter.prop.2}
\end{equation}
\begin{equation}
h(t)^2 + h(2t)^2 = 1, \qquad t \in \big[ 1/2, 1 \big]. \label{eq:filter.prop.3}
\end{equation}
Given Conditions \eqref{eq:filter.prop.1} and \eqref{eq:filter.prop.2},
Condition \eqref{eq:filter.prop.3} implies the following {\em
partition of unity property} of $h^2$:
\begin{equation*}
\sum_{j=0}^\infty \ffilter{\frac{t}{2^j}}^2 = 1, \qquad t \geq 1,
\end{equation*}
in which for every value of $t\ge 1$ we see that at most two terms
are non-zero. For constructions of needlet filters, see
\cite{WaLGSlWo2017,NaPeWa2006-1,MaEtal2007}. A central role will be
played by the filtered version of the kernel of the orthogonal projector
onto the space $\mathbb{P}_{L}$, in particular for $L = 2^{j}-1$
(cf.~\eqref{eq:filtered.kernel} below),
\begin{align*}
\mathrm{proj}_L(\bsx, \bsz)
& :=
\sum_{\ell=0}^L \sum_{m = 1}^{Z(d,\ell)}  Y_{\ell,m}(\bsx) Y_{\ell,m}(\bsz)\\
&= \sum_{\ell=0}^L Z(d,\ell) \, P_\ell( \bsx \cdot \bsz ) \\
&= \frac{( d )_L}{( d/2 )_L} \, P_L^{(d/2,d/2-1)}( \bsx \cdot \bsz )\quad \bsx, \bsz\in \mathbb{S}^d.
\end{align*}

\subsection{Needlets}
Recall that the needlets $\psi_{j,k}$ and generalised needlets
$\Psi_{j,k}$ as defined in \eqref{eq:need} and \eqref{eq:gen_need} are spherical polynomials of degree $2^j-1$ on
$\mathbb{S}^d$.  They are
band-limited, since only terms with index $\ell \in \{ 2^{j-2}+1,
\dots, 2^{j}-1\}$ are present.
They are also zonal functions centered at $\bsx_{j,k}$ and $\bsX_{j,k}$
respectively.

For a sufficiently smooth filter $h \in C^\kappa( \mathbb{R}_+ )$ with
$\kappa \geq\lfloor\frac{d+1}2\rfloor$ (see~\cite{WaSl2017}, the slightly
stronger assumption $\kappa\geq\lfloor\frac{d+3}2\rfloor$ was used in
\cite{WaLGSlWo2017,NaPeWa2006-2}), the needlets $\psi_{j,k}$ are localised in
the sense of
\begin{equation}\label{eq:psi-localised}
\left| \psi_{j,k} ( \bsx ) \right| \leq \frac{c \, 2^{j \, d}}{\left( 1 + 2^j \, \mathrm{dist}( \bsx, \bsx_{j,k} ) \right)^\kappa}, \qquad \bsx \in \mathbb{S}^d, \, j \geq 0, 1 \leq k \leq n_j,
\end{equation}
where the constant $c$ depends only on $d$, $\kappa$, and the filter $h$,
and "$\mathrm{dist}$" is the geodesic distance; see
\cite{NaPeWa2006-2,WaLGSlWo2017}.

\section{The hybrid implementation}\label{sec:hybrid}

It turns out to be convenient not to implement the method in the full
generality of\eqref{eq:VJgen}, but rather as explained in the
Introduction, to insist that the initial levels, up to say level $J_0$,
be exactly as for true needlets; that is to say, that for levels $j = 0,
1, \ldots, J_0$ the needlet cubature rule be exact for all polynomials of
degree up to $2^{j+1}-1$, and so satisfy \eqref{eq:polyprec}. The most
general form of the approximation can then be recovered by setting
$J_0=-1$.  We shall call $J_0$ the {\em hybrid splitting index}.

We can now write the generalised needlet approximation as

\begin{equation}\label{eq:Vhybrid}
V_J^{{\rm gen}}(f; \bsx ) = \sum_{j=0}^{J_0} \sum_{k=1}^{n_j}\left(f,\psi_{j,k}\right)\psi_{j,k}(\bsx)
+ \sum_{j=J_0+1}^J \sum_{k=1}^{N_j} \left(f,\Psi_{j,k}\right) \Psi_{j,k}(\bsx),
\qquad \bsx \in \mathbb{S}^d.
\end{equation}
Our main theorem can then be stated as follows.
For simplicity, from now on we fix $a_\ell^{(s)} = (1+\ell)^{-2s}$.

\begin{theorem} \label{thm:main} Let $d\ge 2$ and $s^* > d/2$. For
  $f\in \mathbb{H}^s(\mathbb{S}^d)$ with $s\in (d/2,s^*)$, let
  $V_J^{{\rm gen}}(f; \cdot)$ be the approximation to $f$ defined by
  \eqref{eq:Vhybrid}, with $\psi_{j,k}$ and $\Psi_{j,k}$ satisfying
  \eqref{eq:need} and \eqref{eq:gen_need} respectively, and with $\bsx_{j,k}$
  and $w_{j,k}$, for $0\le j \le J$ and $1\le k \le n_j$, being the points and
  positive weights of a cubature rule that is exact for all polynomials in
  $\mathbb{P}_{2^{j+1}-1}$; and for $J_0+1 \le j \le J$ and $1 \le k \le N_j$,
  $\bsX_{j,k}$ and $W_{j,k}$ being the points and positive weights of a
  cubature rule satisfying \eqref{eq:QMCprop} for some $\Cqmc>0$.
 Let $\mathcal{P}_\ell(f)$ denote the $\mathbb{L}_2$
  orthogonal projection of $f$ onto $\mathbb{P}_\ell$.
  For filter smoothness $\kappa \geq \lfloor (d+1)/2 \rfloor$, the $\mathbb{L}_2$ error satisfies
\begin{align} \label{eq:approx.error.bound.0}
\|V_J^{{\rm gen}}(f; \cdot)-f\|_{\mathbb{L}_2} &\le \frac{C'}{2^{J s}} \, \|f\|_{\mathbb{H}^s} \, + \,
\Cqmc\left( \sum_{j=J_0+1}^J \frac{B_j}{N_j^{s/d}} \right)
 \, \|f - \mathcal{P}_{2^{J_0-1}}(f)\|_{\mathbb{L}_2}\\
&\le \left[\frac{C'}{2^{J s}}
\,+\, \frac{\Cqmc}{2^{(J_0-1)s}}\left( \sum_{j=J_0+1}^J
\frac{B_j}{N_j^{s/d}} \right)\right]
 \, \|f\|_{\mathbb{H}^s}\nonumber.
\end{align}
Here, $C'$ is the constant from \eqref{eq:need_error},
while $B_j$ is given by
\begin{equation}\label{eq:Bj}
B_j^2 := \int_{\mathbb{S}^d} \| \Lambda_j( \bsx, \cdot ) \Lambda_j( \bsy, \cdot ) \|_{\mathbb{H}^s}^2  \rd \sigma( \bsy ), \qquad J_0 + 1 \leq j \leq J, \;\bsx\in \mathbb{S}^d,
\end{equation}
where the $\Lambda_j$ are ``filtered projection kernels'' defined by
\begin{equation} \label{eq:filtered.kernel}
\Lambda_j( \bsx, \bsz ) := \sum_{\ell=0}^\infty \ffilter{\frac{\ell}{2^{j-1}}} \, Z(d, \ell) \, P_\ell( \bsx \cdot \bsz ), \quad \bsx, \bsz \in \mathbb{S}^d,
\end{equation}
with the filter $h$ as in Section \ref{sec:needletdef}; notice that $B_j$ is independent of $\bsx$. Furthermore, there is a positive constant $c^{\prime\prime}$ such that
\begin{equation} \label{eq:B.j.estimate}
B_j \leq c^{\prime\prime} \, 2^{j(s+d)}.
\end{equation}
\end{theorem}

\begin{proof}
The strategy of the proof is to compare the generalised approximation
$V_{J}^{{\rm gen}}(f,\cdot)$ with the true needlet approximation $V_{J}^{{\rm need}}(f,\cdot)$.

On subtracting \eqref{eq:VJ} from \eqref{eq:Vhybrid} we obtain,
\begin{align}\label{eq:Vdiff}
V_J^{{\rm gen}}(f;\bsx) - V^{{\rm need}}_J(f;\bsx)
&= \sum_{j=J_0+1}^J \left[\sum_{k=1}^{N_j}(f,
\Psi_{j,k})\Psi_{j,k}(\bsx)-\sum_{k=1}^{n_j}(f, \psi_{j,k})\psi_{j,k}(\bsx)\right]\nonumber\\
&= \sum_{j=J_0+1}^J(f,g_j^{{\rm gen}}(\bsx,\cdot)-g_j^{{\rm need}}(\bsx,\cdot)),
\end{align}
where
\begin{equation} \label{eq:gj}
g_j^{{\rm gen}}(\bsx,\bsz) := \sum_{k=1}^{N_j} \Psi_{j,k}( \bsx ) \Psi_{j,k}( \bsz ), \quad
g_j^{{\rm need}}(\bsx,\bsz) := \sum_{k=1}^{n_j} \psi_{j,k}( \bsx ) \psi_{j,k}( \bsz ), \quad
\bsx, \bsz\in \mathbb{S}^d.
\end{equation}
Noting that for $j\ge J_0+1$ both $g_j^{{\rm gen}}(\bsx,\cdot)$ and
$g_j^{{\rm need}}(\bsx,\cdot)$ are orthogonal to all polynomials of degree
less than or equal to $2^{J_0 -1}$, we can replace $f$ on the
right-hand side of \eqref{eq:Vdiff} by $f-\mathcal{P}_{2^{J_0-1}}(f)$. On
using the Cauchy-Schwarz inequality, we then obtain
\begin{align}\label{eq:normfminus}
| V_J^{{\rm gen}}(f;\bsx) - V^{{\rm need}}_J(f;\bsx)|
&\le  \sum_{j=J_0+1}^J \|f-\mathcal{P}_{2^{J_0-1}}(f)\|_{\mathbb{L}_2}
\|g_j^{{\rm gen}}(\bsx,\cdot)-g_j^{{\rm need}}(\bsx,\cdot)\|_{\mathbb{L}_2}\\
&\le  \sum_{j=J_0+1}^J 2^{-(J_0-1)s}\|f \|_{\mathbb{H}^s}
\|g_j^{{\rm gen}}(\bsx,\cdot)-g_j^{{\rm need}}(\bsx,\cdot)\|_{\mathbb{L}_2}\nonumber,
\end{align}
where we used
\begin{align*}
\|f - \mathcal{P}_L(f)\|_{\mathbb{L}_2}^2  
&=
\|\sum_{\ell = L+1}^\infty \sum_{m = 1} ^ {Z(d,\ell)} \widehat{f}_{\ell,m} Y_{\ell,m}(\bsx)\|_{\mathbb{L}_2}^2\\
&=
\sum_{\ell = L+1}^\infty \sum_{m = 1} ^ {Z(d,\ell)} |\widehat{f}_{\ell,m}|^2 \frac{a_\ell^{(s)}}{a_\ell^{(s)}}\\
&\le a_L^{(s)}\|f\|_{H^s}^2.
\end{align*}

Using now the needlet definition \eqref{eq:need}, we can rewrite
$g_j^{{\rm need}}$ as
\begin{align*}
g_j^{{\rm need}}(\bsx, \bsy)&= \sum_{k=1}^{n_j} w_{j,k}
\Lambda_j(\bsx, \bsx_{j,k}) \Lambda_j(\bsy, \bsx_{j,k})\\
&= \int_{\mathbb{S}^d}\Lambda_j(\bsx, \bsz) \Lambda_j(\bsy, \bsz) \rd \sigma(\bsz),
\end{align*}
where in the second step we used the fact that $\Lambda(\bsx, \bsz)\Lambda(\bsy, \bsz)$ is a spherical polynomial in $\bsz$ of degree $2^{j+1}-2$, together with the needlet cubature property \eqref{eq:polyprec}.  In a similar way we have
\begin{equation*}
g_j^{{\rm gen}}(\bsx, \bsy)= \sum_{k=1}^{n_j} W_{j,k}
\Lambda_j(\bsx, \bsX_{j,k}) \Lambda_j(\bsy, \bsX_{j,k}),
\end{equation*}
which is just the generalised cubature approximation to the integral
above, allowing us to write
\begin{equation*}
g_j^{{\rm gen}}(\bsx, \bsy)-g_j^{{\rm need}}(\bsx, \bsy)= \left(\mathrm{Q}[\mathcal{X}_{N_j}] - \mathrm{I}\right)
\left(\Lambda_j(\bsx, \cdot) \Lambda_j(\bsy, \cdot)\right),
\end{equation*}
where $\mathcal{X}_{N_j}$ denotes the point and weight set $\{X_{j,k},W_{j,k}\}$ of the $j$th generalised cubature rule.
It therefore follows from the assumed property \eqref{eq:QMCpropj}
of the sequence of those rules that
\begin{equation*}
|g_j^{{\rm gen}}(\bsx, \bsy)-g_j^{{\rm need}}(\bsx, \bsy)|
\le \frac{\Cqmc}{N_j^{s/d}}\left\|\Lambda_j(\bsx, \cdot) \Lambda_j(\bsy, \cdot)\right\|_{\mathbb{H}^s},
\end{equation*}
and hence
\begin{equation*}
\|g_j^{{\rm gen}}(\bsx, \cdot)-g_j^{{\rm need}}(\bsx, \cdot)\|_{\mathbb{L}_2}
\le \frac{\Cqmc}{N_j^{s/d}}\sqrt{\int_{\mathbb{S}^d}
\left\|\Lambda_j(\bsx, ) \Lambda_j(\bsy, \cdot)\right\|^2_{\mathbb{H}^s}\rd \sigma(\bsy)}.
\end{equation*}
It now follows from \eqref{eq:normfminus} that
\begin{equation}\label{eq:normfminusfunal}
\|V_J^{{\rm gen}}(f,\cdot) - V_J^{{\rm need}}(f,\cdot)\|_{\mathbb{L}_2} \le
\sum_{j=J_0+1}^J\frac{\Cqmc}{N_j^{s/d}}2^{-(J_0-1)s}B_j\|f \|_{\mathbb{H}^s},
\end{equation}
where $B_j$ is given by \eqref{eq:Bj}.  The fact that $B_j$ is independent
of $\bsx$ is a consequence of the rotational invariance of the measure
$\sigma$.

 The desired result now follows immediately from
\eqref{eq:normfminusfunal}, the known error \eqref{eq:need_error} of
the needlet approximation, and the triangle inequality. The proof of the estimate \eqref{eq:B.j.estimate} will be given in Subsection~\ref{sec:error-analysis}.
\end{proof}

\subsection{Estimating $B_j$}\label{sec:error-analysis}
In order to complete the proof of Theorem~\ref{thm:main} we need to estimate the
coefficients $B_j$ given by \eqref{eq:Bj}.

Direct computation utilising the addition theorem, the orthonormality
relations for spherical harmonics, and the Funk-Hecke formula shows that
$B_j$ is given by the following Legendre polynomial expression
\begin{equation}
\label{eq:A.j.Form.1}
B_j^2
= \frac{\omega_{d-1}}{\omega_d} \int_{-1}^1 \left( \sum_{\ell=0}^\infty \ffilter{\frac{\ell}{2^{j-1}}}^2 Z(d, \ell) \, P_\ell( t ) \right)^2 \\
S_M(t) \left( 1 - t^2 \right)^{d/2-1} \rd t,
\end{equation}
where $M:= 2^j$ and
\begin{equation} \label{eq:S.M.t}
S_M( t ) := \sum_{n=0}^{2(M-1)} \frac{1}{a_n^{(s)}} \, Z(d,n) \, P_n( t ).
\end{equation}
For proving \eqref{eq:A.j.Form.1} we first note that $\| \Lambda_{\bsx} \Lambda_{\bsy} \|_{\mathbb{H}^s}^2$ depends only on $\bsx \cdot \bsy$ by rotational invariance. Thus we can write $B_j^2$ as the double integral
\begin{equation*}
B_j^2 = \int_{\mathbb{S}^d} \int_{\mathbb{S}^d} \| \Lambda_{\bsx} \Lambda_{\bsy} \|_{\mathbb{H}^s}^2  \rd \sigma( \bsx ) \rd \sigma( \bsy ), 
\end{equation*}
where, using $h_\ell = \ffilter{\frac{\ell}{2^{j-1}}}$, 
\begin{equation*}
\Lambda_{\bsx}( \bsz ) := \Lambda_j( \bsx, \bsz ) =
\sum_{\ell=0}^\infty\sum_{m=1}^{Z(d,\ell)} h_\ell\, Y_{\ell,m}(\bsx) Y_{\ell,m}(\bsz)=:
\sum_{\ell,m} h_\ell \, Y_{\ell,m}(\bsx) Y_{\ell,m}(\bsz).
\end{equation*}
Now
\begin{equation*}
\| \Lambda_{\bsx} \Lambda_{\bsy} \|_{\mathbb{H}^s}^2 = \sum_{n,k} \frac{\left| \widehat{(\Lambda_{\bsx} \Lambda_{\bsy})}_{n,k} \right|^2}{a_n^{(s)}} = \sum_{n,k} \frac{\left| ( \Lambda_{\bsx} \Lambda_{\bsy}, Y_{n,k} ) \right|^2}{a_n^{(s)}},
\end{equation*}
so
\begin{equation*}
B_j^2 = \sum_{n,k} \frac{1}{a_n^{(s)}} \int_{\mathbb{S}^d} \int_{\mathbb{S}^d} \left| ( \Lambda_{\bsx} \Lambda_{\bsy}, Y_{n,k} ) \right|^2 \rd \sigma( \bsx ) \rd \sigma( \bsy ).
\end{equation*}
Writing the square of the inner product as a double integral, interchanging the two double integrals and using orthogonality relations and the addition theorem, we get
\begin{align*}
&\sum_{k}\int_{\mathbb{S}^d} \int_{\mathbb{S}^d} \left| ( \Lambda_{\bsx} \Lambda_{\bsy}, Y_{n,k} ) \right|^2 \rd \sigma( \bsx ) \rd \sigma( \bsy ) \\
&\phantom{equals}= 
\int_{\mathbb{S}^d} \int_{\mathbb{S}^d} \left( \sum_{\ell} h_{\ell}^2 \, Z(d,\ell) P_\ell( \bsz^\prime \cdot \bsz^{\prime\prime} ) \right)^2 Z(d,n) P_n( \bsz^\prime \cdot \bsz^{\prime\prime} ) \rd \sigma( \bsz^{\prime} ) \rd \sigma( \bsz^{\prime\prime} ) \\
&\phantom{equals}= 
\frac{\omega_{d-1}}{\omega_d} \int_{-1}^1 \left( \sum_{\ell} h_{\ell}^2 \, Z(d,\ell) P_\ell( t ) \right)^2 Z(d,n) P_n( t ) \left( 1 - t^2 \right)^{d/2-1} \rd t.
\end{align*}
The last step utilises rotational invariance and the Funk-Hecke formula. Because of \eqref{eq:filter.prop.2}, the integral vanishes for $n > 2( 2^j - 1 )$ due to orthogonality relations for Legendre polynomials and we arrive at \eqref{eq:A.j.Form.1} and \eqref{eq:S.M.t}.

%

The filtered projection kernel $\Lambda_j$ is localised:
\begin{equation*}
  |\Lambda_j(\bsx,\bsy)|\leq\frac{c2^{jd}}{(1+2^j\dist(\bsx,\bsy))^\kappa}, \qquad \bsx,\bsy \in \mathbb{S}^d,
\end{equation*}
since this is the same estimate as \eqref{eq:psi-localised} (see
\cite{WaSl2017,WaLGSlWo2017,NaPeWa2006-2}), given that $\psi_{j,k}$ is
proportional to $\Lambda_j$. From this and \eqref{eq:filtered.kernel} it
follows that
\begin{equation*}
\left|\sum_{\ell=0}^{\infty} h\left(\frac{\ell}{2^{j+1}}\right) \, Z(d,n) \, P_n( t )\right|
\le\frac{c2^{jd}}{(1+2^j\sqrt{2(1-t)})^\kappa},
\qquad t = \bsx \cdot \bsy, \  \bsx,\bsy \in \mathbb{S}^d,
\end{equation*}
where we used $\mathrm{dist}( \bsx, \bsz ) \geq \| \bsx - \bsz \| =
\sqrt{2( 1 - \bsx \cdot \bsz )}$ for $\bsx, \bsz \in \mathbb{S}^d$.

 Using the last estimate and the estimate
\begin{equation*}
  |S_M(t)|\leq S_M(1)=\sum_{n=0}^{2(M-1)} \frac{Z(d,n)}{a_n^{(s)}}\asymp M^{2s+d},
\end{equation*}
 we arrive at
\begin{equation*}
B_j^2 \leq S_M( 1 ) \frac{\omega_{d-1}}{\omega_d} \int_{-1}^1 \frac{c^2 M^{2d}}{\left( 1 + M \sqrt{2(1-t)} \right)^{2\kappa}} \left( 1 - t^2 \right)^{d/2-1} \rd t.
\end{equation*}

The change of variable $2 (x/M)^2 = 1 - t$, $x \geq 0$, and $(4 x / M^2) \rd x = - \rd t$ gives
\begin{equation}\label{eq:A_j-est}
  \begin{split}
    B_j^2 &\leq
    c^\prime  \, S_M( 1 ) \, M^d \, \int_0^M \frac{x^{d-1}}{\left( 1 + 2 x \right)^{2\kappa}} \left( 1 - \left( \frac{x}{M} \right)^2 \right)^{d/2-1} \rd x \\
    &\leq
    c^\prime \, S_M( 1 ) \, M^d \, \int_0^\infty \frac{x^{d-1}}{\left( 1 + 2 x \right)^{2\kappa}} \, \rd x \leq (c^{\prime\prime})^2 \, 2^{2j(s+d)}, \qquad J_0 + 1 \leq j \leq J.
  \end{split}
\end{equation}
Consequently,
\begin{equation}
\sum_{j=J_0+1}^J \frac{B_j}{N_j^{s/d}} \leq c^{\prime\prime} \sum_{j=J_0+1}^J \frac{( 2^j )^{s+d}}{N_j^{s/d}}.
\end{equation}

\subsection{Exact computations}\label{sec:exact-computations}
Here we provide an exact formula for
$B_j$. This also shows that the estimates given in the last section exhibit the
precise order of magnitude.

Again we set $M = 2^j$.
Expansion of the square in \eqref{eq:A.j.Form.1} gives
\begin{equation*}
B_j^2 = \sum_{n=0}^{2(M-1)} \sum_{\ell=M/4}^{M-1} \sum_{\ell^\prime=M/4}^{M-1} \frac{1}{a_n^{(s)}} \ffilter{\frac{\ell}{2^{j-1}}}^2 \ffilter{\frac{\ell^\prime}{2^{j-1}}}^2 b_{\ell,\ell^\prime,n},
\end{equation*}
where($\lambda=(d-1)/2$)
\begin{align*}
b_{\ell,\ell^\prime,n}
&:= \frac{\omega_{d-1}}{\omega_d} \int_{-1}^1 Z(d, \ell) P_\ell( t ) \, Z(d, \ell^\prime) P_{\ell^\prime}( t ) \, Z(d,n) P_n( t ) \left( 1 - t^2 \right)^{d/2-1} \rd t \\
&= \frac{\ell+\lambda}{\lambda} \, \frac{\ell^\prime+\lambda}{\lambda} \, \frac{n+\lambda}{\lambda} \, \frac{\omega_{d-1}}{\omega_d} \int_{-1}^1 C_{\ell}^\lambda( t ) \, C_{\ell^\prime}^\lambda( t ) \, C_{n}^\lambda( t ) \left( 1 - t^2 \right)^{\lambda-1/2} \rd t.
\end{align*}
By \cite[Corollary 6.8.4, p. 321]{AnAsRo1999},
\begin{equation*}
\begin{split}
&\int_{-1}^1 C_{\ell}^\lambda( t ) \, C_{\ell^\prime}^\lambda( t ) \, C_{n}^\lambda( t ) \left( 1 - t^2 \right)^{\lambda-1/2} \rd t \\
&\phantom{equals}=
\frac{2^{1-2\lambda} \pi}{( \Gamma( \lambda ) )^4} 
\frac{\Gamma(r + 2 \lambda )}{\Gamma( r + \lambda +1)}
\frac{\Gamma(r-n+\lambda)}{\Gamma(r-n+1)}
\frac{\Gamma(r-\ell+\lambda)}{\Gamma(r-\ell+1)} \frac{\Gamma(r-\ell^\prime+\lambda)}{\Gamma(r-\ell^\prime+1)},
\end{split}
\end{equation*}
where $\ell + \ell^\prime + n = 2r$ is even and the sum of any two of $\ell$, $\ell^\prime$, $n$ is not less than the third. The integral vanishes in all other cases.
Hence, we get as sum of positive terms
\begin{equation}\label{eq:A_j-final}
\begin{split}
B_j^2
&=
\frac{\lambda}{\Gamma(2\lambda) ( \Gamma(\lambda) )^2} \mathop{\sum_{n=0}^{2(2^j-1)} \sum_{\ell=2^{j-2}+1}^{2^j-1} \sum_{\ell^\prime=2^{j-2}+1}^{2^j-1}}_{\substack{n + \ell + \ell^\prime = 2r \ \text{even}, \\ \ell + \ell^\prime \geq n, \, n + \ell \geq \ell^\prime, \, \ell^\prime + n \geq \ell}} \frac{\Gamma(r+2\lambda)}{\Gamma(r+\lambda+1)} \, \frac{1}{a_n^{(s)}}  \frac{n + \lambda}{\lambda} \, \frac{\Gamma(r-n+\lambda)}{\Gamma(r-n+1)} \\
&\phantom{equalsequalseq}\times \ffilter{\frac{\ell}{2^{j-1}}}^2  \, \frac{\ell + \lambda}{\lambda} \, \frac{\Gamma(r-\ell+\lambda)}{\Gamma(r-\ell+1)} \, \ffilter{\frac{\ell^\prime}{2^{j-1}}}^2 \, \frac{\ell^\prime + \lambda}{\lambda} \, \frac{\Gamma(r-\ell^\prime+\lambda)}{\Gamma(r-\ell^\prime+1)}.
\end{split}
\end{equation}
A crude estimation
of the exact formula \eqref{eq:A_j-final} yields a lower bound of order $2^{2j(s+d)}$, which is the same order as obtained for the upper bound in \eqref{eq:A_j-est}.
%
More precisely, we fix 
$3/4 < q < 1$ such that 
\begin{equation*}
 \min_{q \leq x \leq 2-q} h(x)^2\geq\frac12.
\end{equation*}
Then $\ffilter{\ell / 2^{j-1}}^2 \geq \frac12$ 
for $q2^{j-1}  \leq \ell \leq (2-q)2^{j-1} $. 
The bound $q > 3/4$ ensures that the summation indices $n, \ell, \ell^\prime$ satisfy the inequalities $\ell + \ell^\prime \geq n$, $ n + \ell \geq \ell^\prime$, and $\ell^\prime + n \geq \ell$, for all $(n,\ell,\ell^\prime)$ in the cube $\Cube=[q2^{j-1},(2-q)2^{j-1}]^3 \cap \mathbb{Z}^3$. Furthermore, the ratios of gamma functions can be estimated using the asymptotic relations
\begin{equation*}
\frac{\Gamma( x + a )}{\Gamma( x + b )} \sim x^{a-b} \qquad \text{as $x \to \infty$.}
\end{equation*}
Taking into account \eqref{eq:sequenceassumption}, we get 
\begin{align*}
B_j^2 
&\geq 
c^\prime
\sum_{\substack{(n, \ell, \ell^\prime)\in \Cube\\n + \ell + \ell^\prime = 2r \ \text{even}}} r^{\lambda-1} n^{2s+1} \left( r - n \right)^{\lambda-1} \ell \left( r - \ell \right)^{\lambda-1} \ell^\prime \left( r - \ell^\prime \right)^{\lambda-1} \\
&\geq 
c^{\prime\prime} \left( 2^j \right)^{2s+2d-3} \sum_{\substack{(n, \ell, \ell^\prime)\in \Cube \\
n + \ell + \ell^\prime \ \text{even}}} 1.
\end{align*}
Here we have used the fact that by the choice of the summation range $\Cube$ all terms $r,n,r-n,\ell,r-\ell,\ell^\prime,r-\ell^\prime$ are of order $2^j$.
The right-most triple sum counts the number of integer triples in a box with side length $(1-q)2^{j-1}$ having even $\|\cdot\|_1$-distance from the origin. This number is of order $(2^j)^3$ which yields the desired result for $B_j^2$. 

\section{Numerical experiments}

For the following experiments, the squared $\mathbb{L}_2$ error in approximating a function $f$
\[
  \| V_J^{\mathrm{gen}}(f;\cdot) - f \|_{\mathbb{L}_2(\mathbb{S}^2)}^2 =
  \int_{\mathbb{}S^2} \left|f(\bsx) - V_J^{\mathrm{gen}}(\bsx)\right|^2 \mathrm{d} \sigma(\bsx)
\]
is estimated using $10^6$ equal area points~\cite{RaSaZh1994} on $\mathbb{S}^2$ with equal weights.

We illustrate the behaviour of the approximation using two test functions on $\mathbb{S}^2$, one with unlimited smoothness and the second with a parameter that can be used to control the Sobolev space smoothness.
The first is the Franke function (see Renka~\cite[p. 146]{Renka1988}),
\begin{eqnarray*}
f(x, y, z) & = & 0.75 \exp(-(9x - 2)^2/4 - (9y - 2)^2/4 - (9z - 2)^2/4) \\
           & &  + 0.75 \exp(-(9x + 1)^2/49 - (9y + 1)/10 - (9z + 1)/10) \\
           & &  + 0.5 \exp(-(9x - 7)^2/4 - (9y - 3)^2/4 - (9z - 5)^2/4) \\
           & &  - 0.2 \exp(-(9x - 4)^2 - (9y - 7)^2 -(9z - 5)^2), \qquad (x, y, z) \in \mathbb{S}^2,
\end{eqnarray*}
which is in $C^\infty(\mathbb{S}^2)$,
and hence all Sobolev spaces $\mathbb{H}^s(\mathbb{S}^d)$, $s \geq d/2$, but the Gaussian with negative coefficient and 
centre at $(4/9, 7/9, 5/9) \in \mathbb{R}^3$, close to the unit sphere,
is hard to approximate.

The second family of test functions are the sums of $6$ compactly supported Wendland radial
basis functions~\cite{WaLGSlWo2017}
\begin{equation}
f_k(\bsx) = \sum_{j=1}^6 \phi_k(\bsz_j - \bsx), \quad k \geq 0,
\end{equation}
where $\bsz_1:=(1, 0, 0)$, $\bsz_2:=(-1, 0, 0$), $\bsz_3:=(0, 1, 0)$, $\bsz_4:=(0, -1, 0)$, $\bsz_5:=(0, 0, 1)$, $\bsz_6:=(0, 0, -1)$.
 The original Wendland functions~\cite{Wendland1995} are
 \begin{equation}
\widetilde{\phi}_k(r) :=
\left\{
\begin{array}{ll}
(1 - r)_+^2,  &  k = 0, \\[1ex]
(1 - r)_+^4 (4r + 1), & k= 1, \\[0.5ex]
(1 - r)_+^6 (35r^2 + 18r + 3)/3, & k= 2, \\[0.5ex]
(1 - r)_+^8 (32r^3 + 25r^2 + 8r + 1), & k= 3, \\[0.5ex]
(1 - r)_+^{10} (429r^4 + 450r^3 + 210r^2 + 50r + 5)/5, & k= 4,
\end{array}\right.
\end{equation}
where $(r)_+:= \max\{r, 0\}$ for $r\in\mathbb{R}$.
The normalized (equal area) Wendland functions as defined in [5] are
\[
 \phi(r) := \widetilde{\phi}_k\left(\frac{r}{\delta_k}\right),
\qquad
\delta_k := \frac{3(k + 1)\Gamma(k + \tfrac{1}{2})}{2 \Gamma(k + 1)}, \quad k \geq 0.
\]
Scaled in this way, the Wendland functions converge pointwise to a Gaussian as $k\to\infty$,
see Chernih et al. \cite{ChSlWo2014}.
Thus as $k$ increases the main change is to the smoothness of $f$,
with~\cite{NaWa2002, LeGeSlWe2010} $f_k \in  \mathbb{H}^{k + 3/2}(\mathbb{S}^2)$.
Other possible test functions with known Sobolev space smoothness can be found in \cite{Brauchart2018}.

For levels $j = 0, \ldots, J_0$, the number of needlets $n_j$ at level $j$ is determined by the requirement \eqref{eq:polyprec} that the needlet quadrature rule has degree of precision $L_j = 2^{j+1}-1$.
A (generally not achievable) lower bound,
see, for example \cite{Cools1997} or \cite{HeSlWo2015}, on the number of points to satisfy \eqref{eq:polyprec} is
\[
  n_j \geq \mathrm{dim}\left( \mathbb{P}_{\floor{\tfrac{L_j}{2}}} (\mathbb{S}^d) \right) =  Z\left(d+1,\floor{\tfrac{L_j}{2}}\right) \sim \frac{2}{\Gamma(d+1)} \left(\floor{\tfrac{L_j}{2}}\right)^d.
\]
For $d = 2$, needlet cubature rules with degree of precision $L_j$ and $n_j \approx \frac{L_j^2}{2}$ (roughly twice the lower bound) can be achieved by positive weight Gauss-Legendre rules (which have poor separation close to the poles), see for example~\cite{HeSlWo2015}, or equal weight spherical designs~\cite{Womersley2018}.
In this case $n_j \approx 2^{2j+1}$.

Figure~\ref{F:FrankeFull} illustrates the rapid decay of the $\mathbb{L}_2$ error for a full needlet approximation (with $J = J_0 = 7$) using symmetric spherical designs (SS), see \cite{Womersley2018},
of the Franke function.

For levels $j = J_0 + 1, \ldots,J$ there is additional flexibility in the choice of the generalised needlet cubature rule.
\begin{figure}[ht]
\centering
\includegraphics[width=0.8\textwidth, trim = 0.72cm 0.0cm 0.72cm 0.2cm, clip]{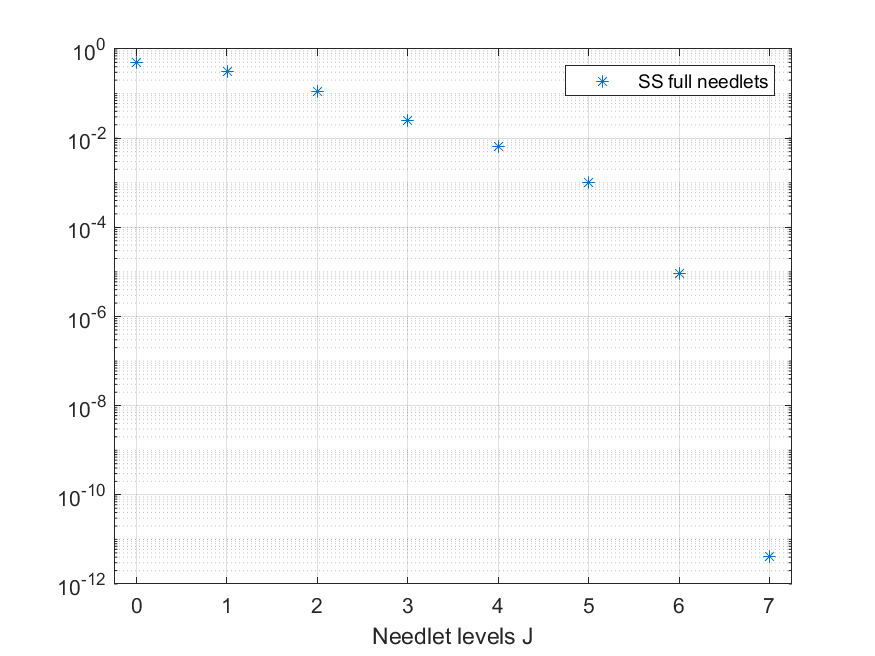}
\caption{$\mathbb{L}_2$ errors approximating the Franke function using a full needlet approximation with symmetric spherical designs.}
\label{F:FrankeFull}
\end{figure}

\begin{figure}
\begin{subfigure}[b]{0.45\textwidth}
  \centering 
  \includegraphics[width=0.95\textwidth, trim = 1.7cm 6.5cm 1.7cm 6.5cm, clip]{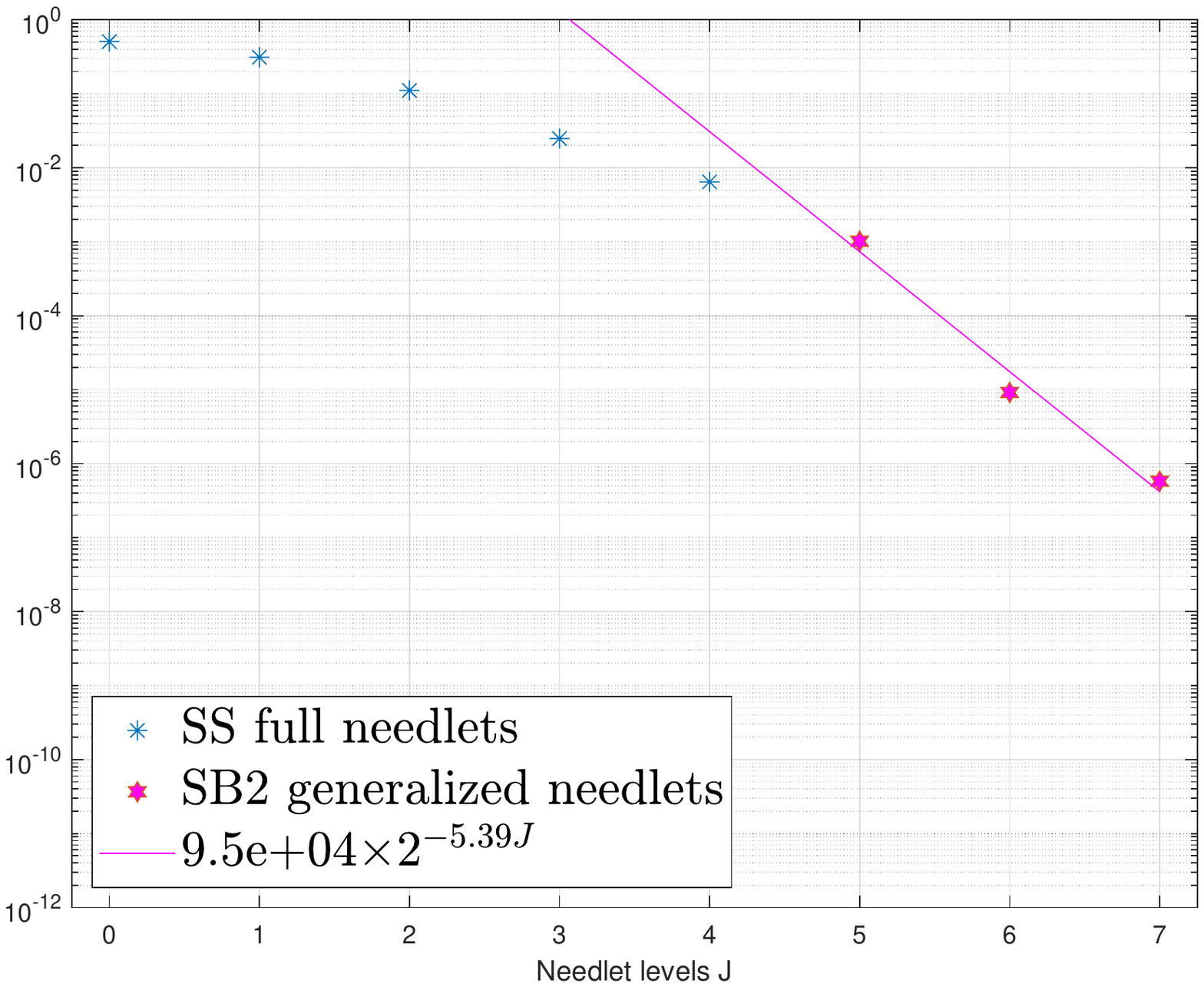}
\caption[]{\label{F:Franke1}generalised  needlets, $N_j = 2^{2(j+1)}$}
\end{subfigure}
\begin{subfigure}[b]{0.45\textwidth}
  \centering  
  \includegraphics[width=0.95\textwidth, trim = 1.7cm 6.5cm 1.7cm 6.5cm, clip]{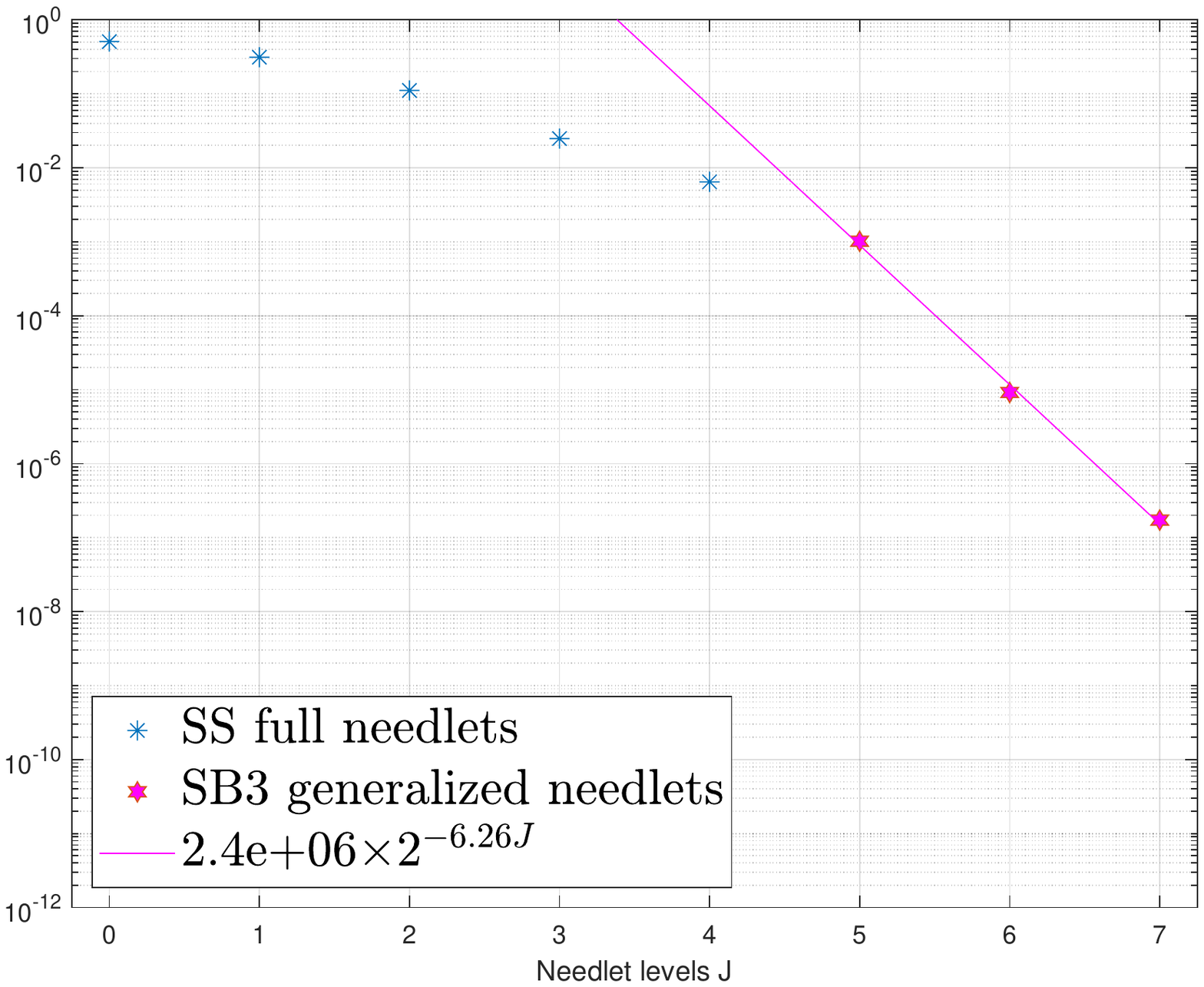}
  \caption[]{\label{F:Franke2}generalised  needlets, $N_j = 2 \times 2^{2(j+1)}$}
\end{subfigure}
\vskip\baselineskip
\begin{subfigure}[b]{0.45\textwidth}
  \centering  
  \includegraphics[width=0.95\textwidth, trim = 1.7cm 6.5cm 1.7cm 6.5cm, clip]{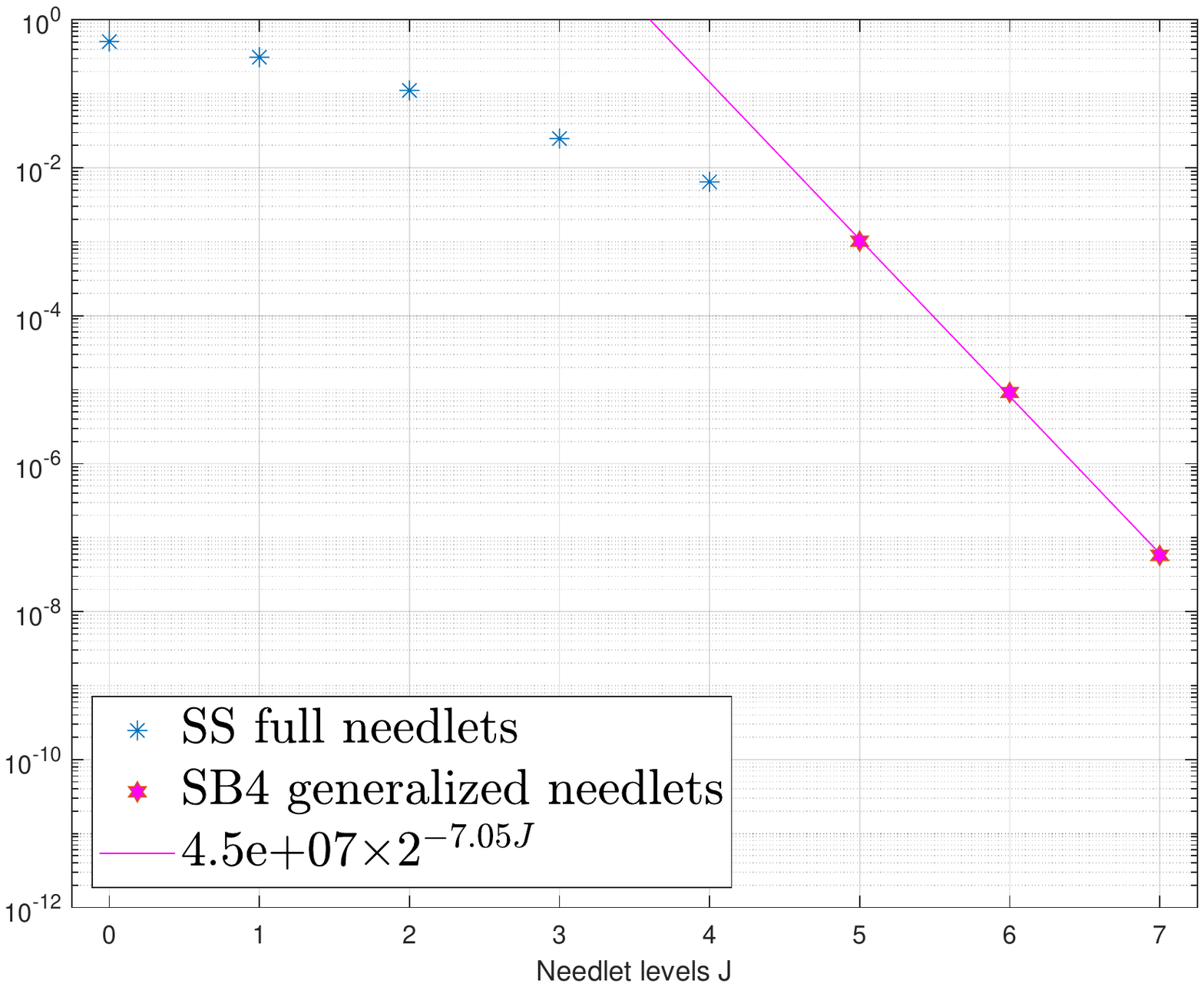}
  \caption[]{\label{F:Franke3}generalised  needlets, $N_j = 4 \times 2^{2(j+1)}$}
\end{subfigure}
\begin{subfigure}[b]{0.45\textwidth}
  \centering
  \includegraphics[width=0.95\textwidth, trim = 1.7cm 6.5cm 1.7cm 6.5cm, clip]{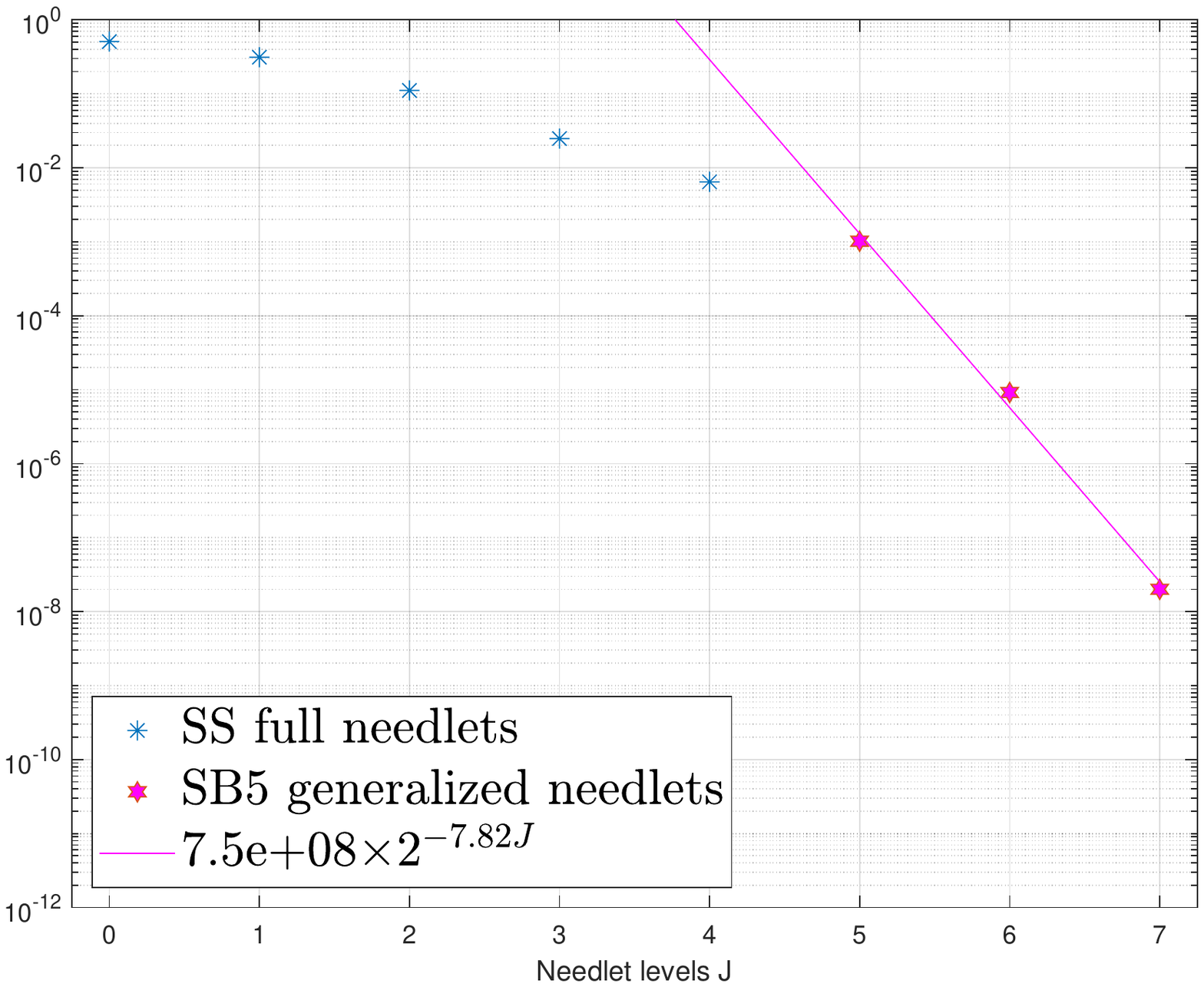}
  \caption[]{\label{F:Franke4}generalised  needlets, $N_j = 8 \times 2^{2(j+1)}$
  }
\end{subfigure}
\caption{$L_2$ errors for approximating the Franke function with a full needlet using symmetric spherical designs for levels $0,\ldots,4$ and generalised needlets for levels $5, 6, 7$ using generalised spiral points with $N_j = c \times 2^{2(j+1)}$ and $c = 1, 2, 4 ,8$.}
\label{F:Franke}
\end{figure}
Figure~\ref{F:Franke} illustrates the $\mathbb{L}_2$ errors for approximating the Franke function with a generalised needlet with $J_0 = 4$ and $J = 7$. Again symmetric spherical designs are used for the full needlets up to level $J_0$. The Bauer~\cite{Bauer2000} version of the generalised spiral points \cite{RaSaZh1994} were used for the generalised needlets in levels $j = 7, 8, 9$.
The four subplots use
$N_j = c 2^{2(j+1)}$ for $c = 1, 2, 4, 8$,
illustrating improving convergence for
the generalised needlet levels $j = 5, 6, 7$ as $N_j$ increases.

\begin{figure}
    \begin{subfigure}[b]{0.45\textwidth}
    \centering 
    \includegraphics[width=0.95\textwidth,trim = 1.7cm 6.5cm 1.7cm 6.5cm, clip]{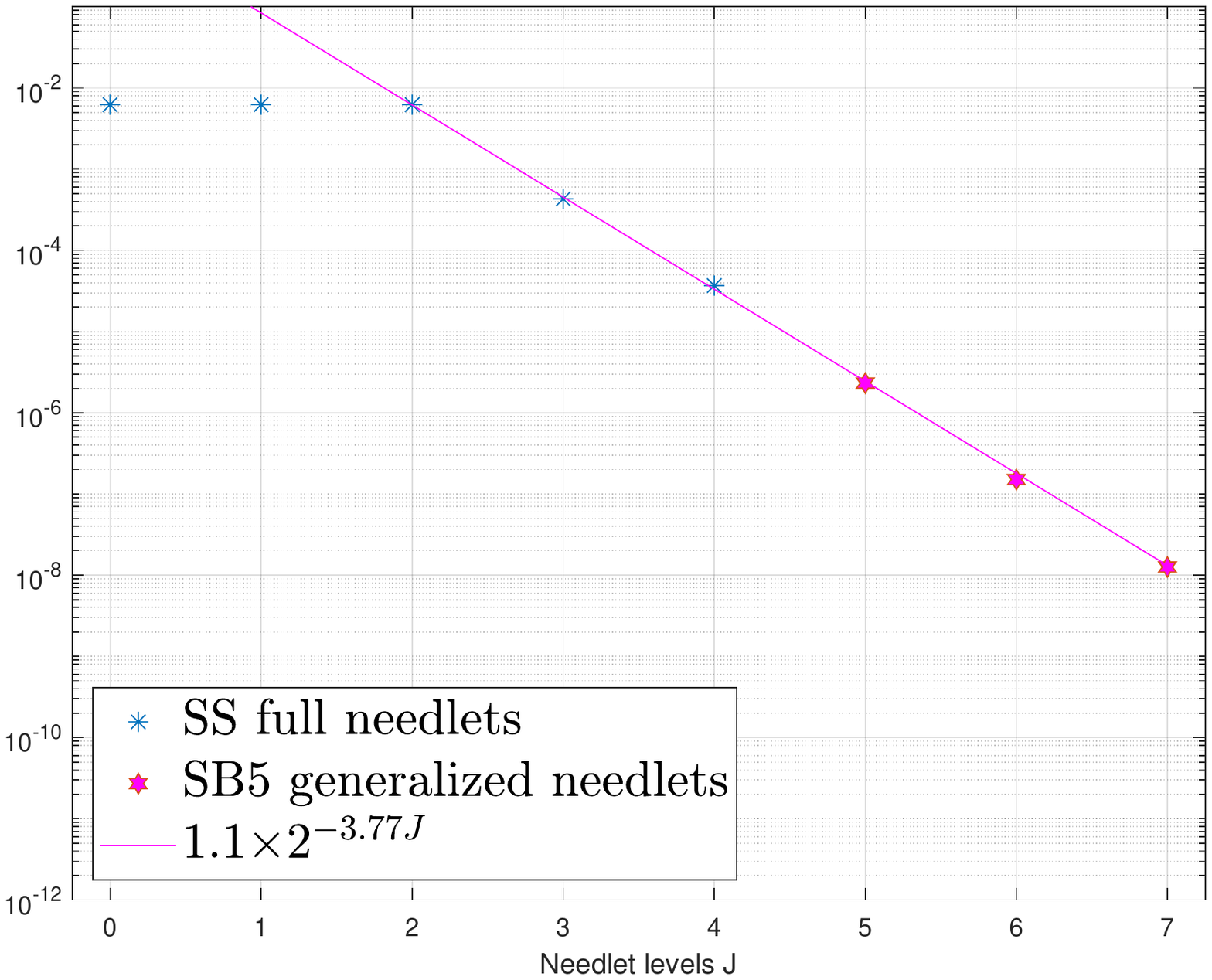}
    \caption[]{Wendland test function $f_1$}
    \label{F:Wendland1}
    \end{subfigure}
    \begin{subfigure}[b]{0.45\textwidth}
    \centering         
    \includegraphics[width=0.95\textwidth,trim = 1.7cm 6.5cm 1.7cm 6.5cm, clip]{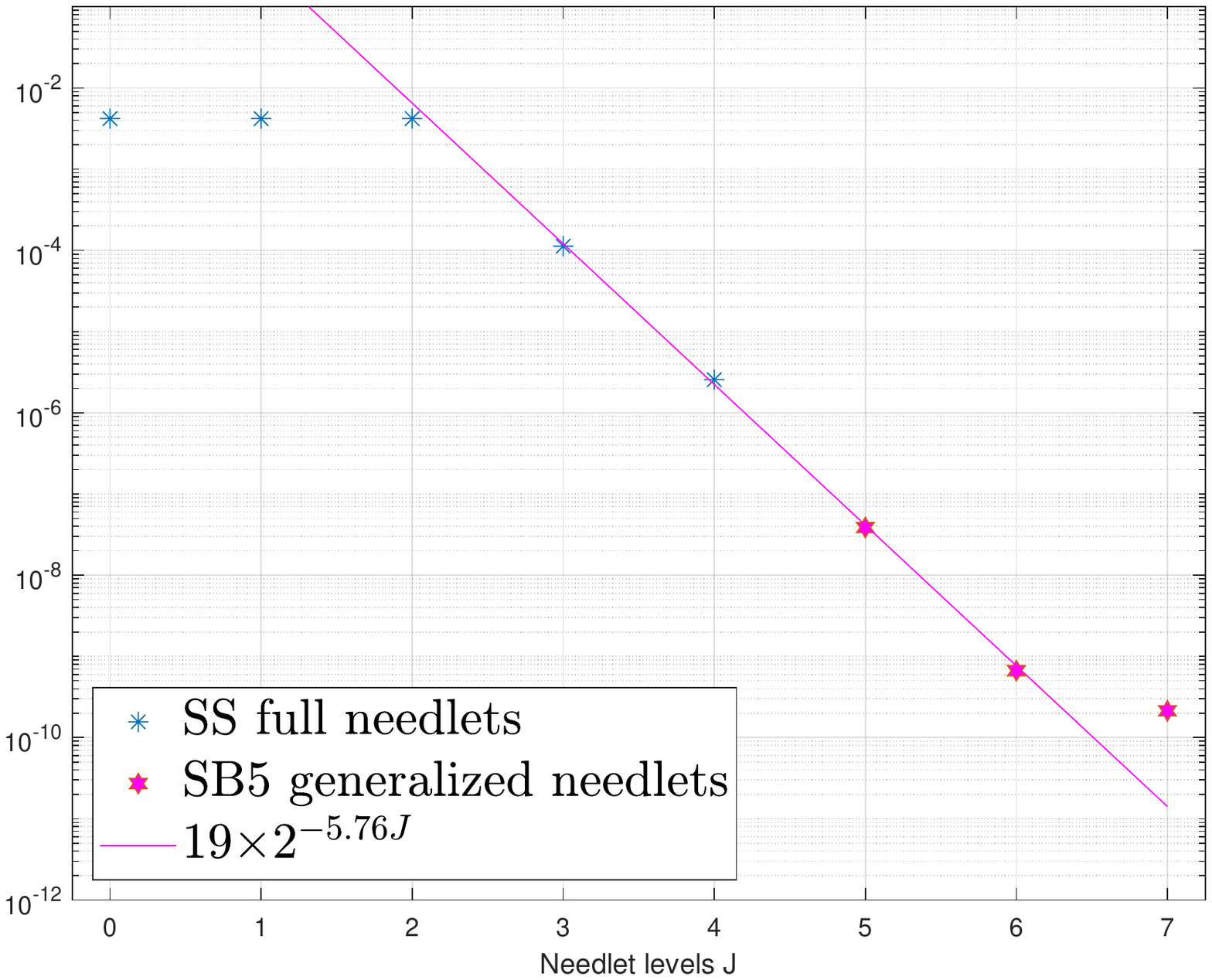}
    \caption[]{Wendland test function $f_2$}
    \label{F:Wendland2}
    \end{subfigure}

    \vskip\baselineskip
    \begin{subfigure}[b]{0.45\textwidth}
    \centering         
    \includegraphics[width=0.95\textwidth,trim = 1.7cm 6.5cm 1.7cm 6.5cm, clip]{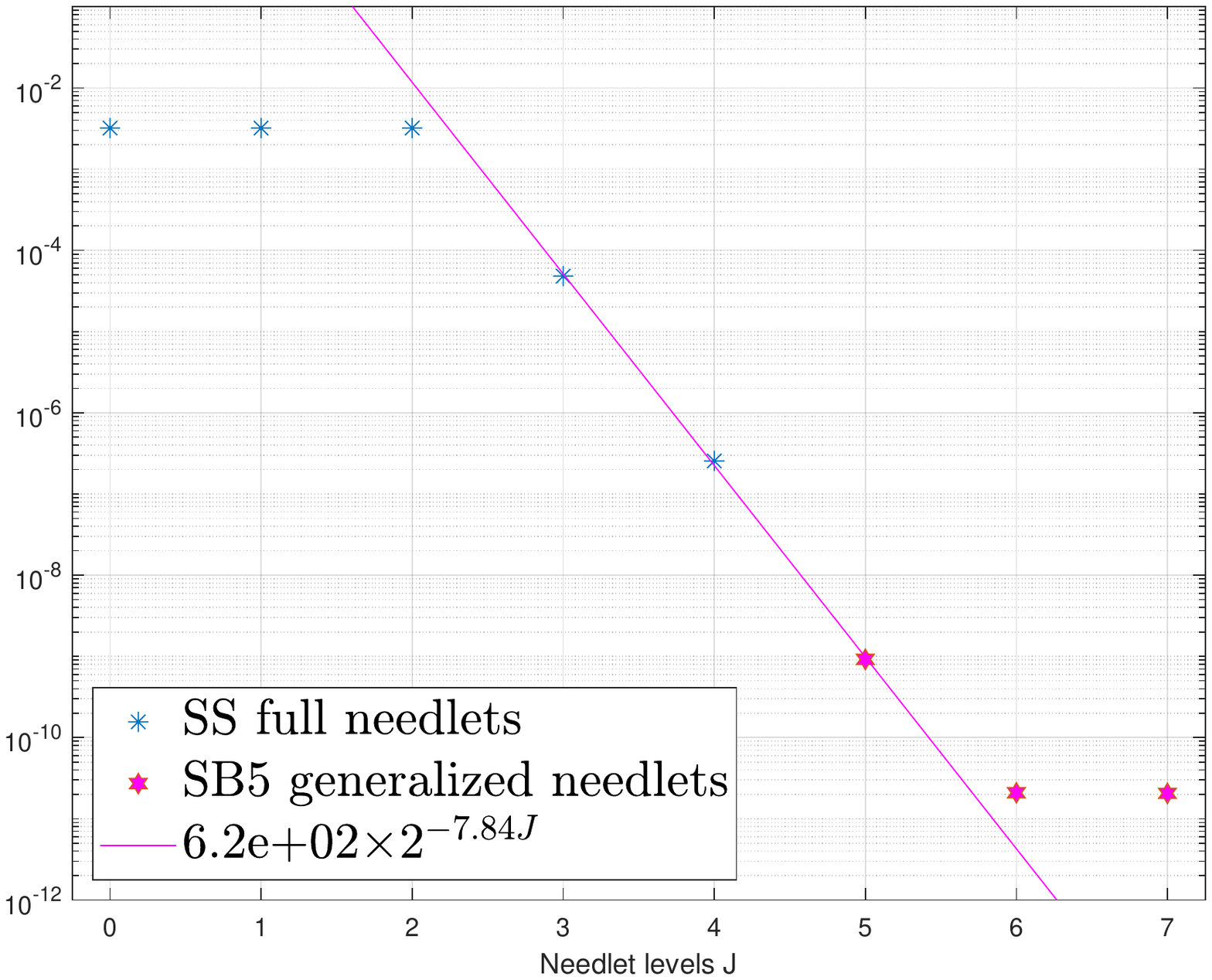}
    \caption[]{Wendland test function $f_3$}
    \label{F:Wendland3}
    \end{subfigure}
    \begin{subfigure}[b]{0.45\textwidth}
    \centering
    \includegraphics[width=0.95\textwidth,trim = 1.7cm 6.5cm 1.7cm 6.5cm, clip]{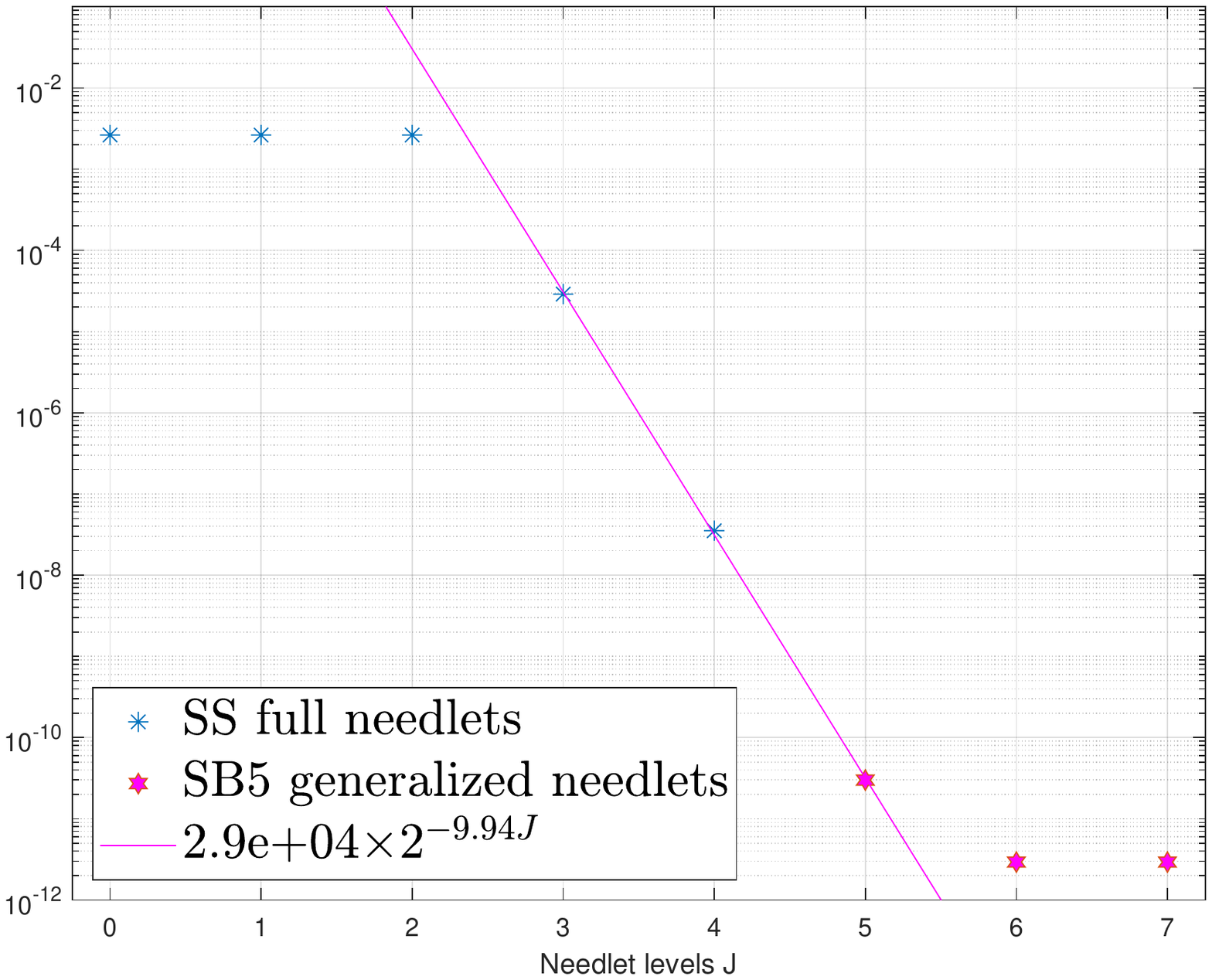}
    \caption[]{Wendland test function $f_4$}
    \label{F:Wendland4}
    \end{subfigure}
\caption{$L_2$ errors when approximating the Wendland test functions $f_k$, $k = 1, 2, 3, 4$ using full needlets with symmetric spherical designs for levels $0,\ldots,4$ and generalised needlets using spiral points with $N_j = 8 \times 2^{2(j+1)}$ for levels $5, 6, 7$.}
\label{F:WendlandSB}
\end{figure}        
        
Figure~\ref{F:WendlandSB} illustrates the importance of the smoothness of the function to be approximated using
the sum-of-Wendland functions $f_k$ for $k = 1, 2, 3, 4$
in subfigures~\ref{F:Wendland1} to \ref{F:Wendland4}. 
The full needlets for levels $j = 0, \ldots, 4$
again used the symmetric spherical designs,
while the generalised needlets used the Bauer generalised spiral points with $N_j = 8 \times 2^{2(j+1)}$.
The estimated order of convergence increases from
$2^{-3.7 J}$ for $f_1$ to  $2^{-9.9 J}$ for $f_4$.
For the smoother functions $f_3$ and $f_4$, accumulation of rounding errors limited the achievable accuracy for levels $6$ and $7$.

\subsection{Discussion}
 The error bound \eqref{eq:approx.error.bound.0}
suggests that at least $N_j > B_j^\frac{d}{s}$.
Using the estimate $B_j \approx 2^{j(s+d)}$,
this gives
$N_j > 2^{j (d + d^2/s)} = 2^{j (2+4/s)}$ for $d = 2$.
Table~\ref{T:Bj} gives numerically calculated values of $B_j$ using \eqref{eq:A.j.Form.1} for $d = 2$, $s = 2, 3$ and levels $j = 5, 6, 7$, along with the asymptotic estimate $2^{j(s+d)}$.
Keeping  $B_j / N_j^{s/d} < 1$ then requires
$N_j > B_j^{d/s}$  which gives, for $d = 2$, $N_j > B_j$ for $s = 2$ and $N_j > B_j^{2/3}$ for $s = 3$. Using the numerically calculated values for $B_j$ in Table~\ref{T:Bj} suggest values
of $N_j$ considerably larger than those
used in Figure~\ref{F:Franke},
where subplot~\ref{F:Franke4} has $N_5 = 32,768$,
$N_6 = 131,072$ and $N_7 = 524,288$.
On the other hand the symmetric spherical designs in Figure~\ref{F:FrankeFull} required just $n_5 = 2,018$, $n_6 = 8,130$ and $n_7 = 32,642$ needlets at levels $5, 6$ and $7$.

\begingroup
\renewcommand{\arraystretch}{1.3} 
\begin{table}
\centering
\begin{tabular}{c|c|lll}
$s$ & & $j = 5$ & $j = 6$ & $j = 7$ \\
\hline
$2$ & $B_j$ & $1.92 \times 10^5$ & $2.84 \times 10^6$ & $4.36 \times 10^7$ \\
$2$ & $2^{j(s+d)}$ & $1.05 \times 10^6$ & $1.68 \times 10^7$ & $2.68 \times 10^8$ \\
\hline
$3$ & $B_j$ & $5.48 \times 10^6$ & $1.59 \times 10^8$ & $4.84 \times 10^9$ \\ 
$3$ & $2^{j(s+d)}$ & $3.36 \times 10^7$ & $1.07 \times 10^9$ & $3.44 \times 10^{10}$ \\
\hline
\end{tabular}
\caption{\label{T:Bj}Numerically calculated values of $B_j$ using \eqref{eq:Bj} and the asymptotic estimate $2^{j(s+d)}$ for $d = 2$, $s = 2, 3$ and levels $j = 5, 6, 7$.}
\end{table}
\endgroup

A key step in a implementation of a (generalised) needlet approximation is
the evaluation of the continuous inner products \eqref{eq:inner}
to obtain the needlet coefficients
$(f, \psi_{j,k})$ or $(f, \Psi_{j,k})$.
A useful strategy, especially if the (generalised) needlet approximations is to be evaluated at many points $\bsx\in\mathbb{S}^d$, is to first evaluate
all the
$\sum_{j=0}^{J_0} n_j + \sum_{j=J_0+1}^J N_j$ needlet coefficients.

It is shown in \cite{WaLGSlWo2017} that approximating the level $j$ continuous inner products
by a quadrature rule that is exact for polynomials of degree $3\times 2^j - 1$ preserves the
error estimates for a full needlet approximation.
It is also possible to exploit the localization properties \eqref{eq:psi-localised} of needlets,
for example by integrating over a spherical cap around each needlet center~\cite{HeWo2012}.
To focus on flexibility gained from using generalised needlets, we 
have assumed in this section that all the needlet coefficient inner products are evaluated to high accuracy. In particular, a high order Gauss-Legendre (GL) rule was used for the results in Figures~\ref{F:FrankeFull} to \ref{F:WendlandSB}.

All the numerical results in this Section used a filter $h$ with smoothness $\kappa = 5$, see \eqref{eq:filter.prop.1}.
As the support of the filter $h$ is $[\frac{1}{2}, 2]$ and the continuity of $h$ implies $h(\frac{1}{2}) = 0 = h(2)$, evaluation of each generalised needlet $\Psi_{j,k}$ 
requires evaluating $P_\ell(\bsX_{j,k}\cdot \bsx)$
for $3\times 2^{j-2} - 1$ degrees $\ell = 2^{j-2}+1,\ldots,2^j-1$.
Assuming the inner products $(f, \Psi_{j,k})$ are known,
then evaluating the generalised needlet approximation
has a cost
\[
  \sum_{j=1}^J N_j \times \mbox{cost}(\Psi_{j,k})
  \approx 3 \sum_{j=1}^J N_j \sum_{\ell=2^{j-2}}^{2^j} \mbox{cost}(P_\ell).
\]
Thus increasing $N_j$ linearly increases the cost of evaluating the generalized needlet approximation.

\section{Conclusions}

In this paper we introduced the notion of generalised needlets on the
sphere, in which the cubature rules used in the construction are not
required to be exact for any polynomials, but
instead need to belong to a sequence of positive weight rules with
prescribed rate of decay for the $\mathbb{L}_2$ cubature error when applied to
functions in Sobolev spaces, thereby extending to unequal weights the
recently introduced concept of QMC designs. In practice a hybrid implementation is recommended, in which the lower levels use traditional needlets, while the generalised needlets are used only for some number of higher levels. The error bound given in
Section \ref{sec:hybrid} and the numerical experiments show that the
construction is feasible even for rules with equal weights and very large
point sets.
\bibliographystyle{abbrv}

\bibliography{LiberatedNeedlet}

\end{document}